\documentclass[a4paper,12pt,reqno]{amsart}

\usepackage[
  margin=30mm,
  marginparwidth=25mm,     
  marginparsep=2mm,       
  bottom=25mm,
  ]{geometry}

\usepackage[T1]{fontenc}
\usepackage[utf8]{inputenc}
\usepackage[english]{babel}
\usepackage{amsmath,amssymb,amsthm,mathtools}

\usepackage{latexsym}
\usepackage{delarray}
\usepackage{bbm}
\usepackage{scrextend}
\usepackage{hyperref}
\usepackage[pdftex,usenames,dvipsnames]{xcolor}
\usepackage{datetime}
\usepackage{mathrsfs}
\usepackage{enumitem}  
\usepackage{tikz}
\usepackage{enumerate}
\usepackage{comment}

\usepackage{geometry}
\newtheorem{Th}{Theorem}[section]

\theoremstyle{definition}
\newtheorem{Lem}[Th]{Lemma}
\newtheorem{Cor}[Th]{Corollary}
\newtheorem{Rem}[Th]{Remark}
\theoremstyle{definition}
\newtheorem{Def}[Th]{Definition}

\numberwithin{equation}{section} 
 \normalfont

\newcommand{\R}{\mathbb{R}}

\renewcommand{\leq}{\leqslant}
\renewcommand{\geq}{\geqslant}

\allowdisplaybreaks 

\makeatletter
\newcommand{\setword}[2]{%
  \phantomsection
  #1\def\@currentlabel{\unexpanded{#1}}\label{#2}%
}
\makeatother

\begin{document}

 \title[Semilinear biharmonic heat equations]
   {On the critical behavior for the semilinear biharmonic heat equation with forcing term in  exterior domain}

\author[N.N. Tobakhanov]{Nurdaulet N. Tobakhanov}
\address{
  Nurdaulet N. Tobakhanov:
  \endgraf
  Department of Mathematics
  \endgraf
Nazarbayev University, Astana, Kazakhstan
  \endgraf
  {\it } {\rm nurdaulet.tobakhanov@nu.edu.kz}
  }

 \author[B.T. Torebek]{Berikbol T. Torebek}
\address{
	Berikbol T. Torebek:
	\endgraf
Institute of
Mathematics and Mathematical Modeling
\endgraf
Shevchenko street 28, 050010 Almaty, Kazakhstan
	\endgraf {\it} {\rm torebek@math.kz}
} 

\thanks{No new data was collected or generated during the course of research.}


     \begin{abstract}In this paper, we investigate the critical behavior of solutions to the semilinear biharmonic heat equation with forcing term $f(x),$ under six homogeneous boundary conditions. This paper is the first since the seminal work by Bandle, Levine, and Zhang [J. Math. Anal. Appl. 251 (2000) 624--648], to focus on the study of critical exponents in exterior problems for semilinear parabolic equations with a forcing term. By employing a method of test functions and comparison principle, we derive the critical exponents $p_{Crit}$ in the sense of Fujita. Moreover, we show that $p_{Crit}=\infty$  if  $N=2,3,4$  and $p_{Crit}=\frac{N}{N-4}$  if $N \geq 5$. The impact of the forcing term on the critical behavior of the problem is also of interest, and thus a second critical exponent in the sense of Lee-Ni, depending on the forcing term is introduced. We also discuss the case $f\equiv 0$, and present the finite-time blow-up results and lifespan estimates of solutions for the subcritical and critical cases. The lifespan estimates of solutions are obtained by employing the method proposed by Ikeda and Sobajama in [Nonlinear Anal. 182 (2019) 57--74]. 
      \end{abstract}
     \maketitle

     \tableofcontents

\section{Introduction}
In this paper, we are concerned with the existence and nonexistence of global weak solutions to the
inhomogeneous semilinear heat equations with forcing terms on exterior domains
\begin{equation}\label{main}
\left\{\begin{array}{lll}
u_t+\Delta^2 u=|u|^{p}+f(x), & \text { in }  D^c \times(0, \infty) , \\{}\\
u(x, 0)=u_0(x), & \text { in } D^c, 
\end{array}\right.
\end{equation}

\noindent where $p > 1$ is a constant, $\Delta$ is the Laplace operator, and $D=\overline{B(0,1)}$ is the closed unit ball, and $D^c$ is its complement in $\R^N.$   The problem \eqref{main} is investigated under the following boundary conditions:

⟨I⟩ \quad  Navier boundary conditions: 
\begin{equation}\label{1.1}
\left\{\begin{array}{lll}
u(x,t)=0,  & \text { in } \partial D \times (0, \infty),\\ {}\\
\Delta u(x,t)=0, & \text { in } \partial D \times (0, \infty); 
\end{array}\right.
\end{equation}

 ⟨II⟩ \quad Dirichlet boundary conditions:
\begin{equation}\label{1.2}
   \left\{\begin{array}{lll}
   u(x,t)=0,  & \text { in } \partial D \times (0, \infty),\\ {}\\
   \frac{\partial u(x,t)}{\partial n} =0, & \text { in } \partial D \times (0, \infty);
   \end{array}\right.
\end{equation}

⟨III⟩ \quad Dirichlet-Navier boundary conditions:
\begin{equation}\label{1.4}
  \left\{\begin{array}{lll} 
   u(x,t)=0,  & \text { in } \partial D \times (0, \infty),\\{}\\ 
   \frac{\partial \Delta u(x,t)}{\partial n} =0, & \text { in } \partial D \times (0, \infty);
   \end{array}\right. 
\end{equation}

⟨IV⟩ \quad Kuttler-Sigillito boundary conditions:
\begin{equation}\label{1.5}
   \left\{\begin{array}{lll} 
   \frac{\partial u(x,t)}{\partial n}=0,  & \text { in } \partial D \times (0, \infty),\\{}\\ \frac{\partial \Delta u(x,t)}{\partial n}=0, & \text { in } \partial D \times (0, \infty);
   \end{array}\right. 
\end{equation}

 ⟨V⟩ \quad Neumann-Navier boundary conditions:
\begin{equation}\label{1.6}
    \left\{\begin{array}{lll} 
   \Delta u(x,t)=0,  & \text { in } \partial D \times (0, \infty),\\{}\\ \frac{\partial \Delta u(x,t)}{\partial n}=0, & \text { in } \partial D \times (0, \infty); \end{array}\right.  
\end{equation}

⟨VI⟩ \quad Neumann boundary conditions:
\begin{equation}\label{1.3}
   \left\{\begin{array}{lll}
   \frac{\partial u(x,t)}{\partial n}=0,  & \text { in } \partial D \times (0, \infty),\\{}\\ \Delta u(x,t)=0, & \text { in } \partial D \times (0, \infty), 
   \end{array}\right.
\end{equation}

\noindent where $n$ denotes the outward unit normal vector on $\partial D$.  For the convenience of the reader, Roman numerals were allocated to each boundary condition. For instance, the problem I indicate semilinear equation \eqref{main} with boundary conditions \eqref{1.1}, etc.

Before stating the main results, let us briefly point out the results related to the problem \eqref{main}. In 1966, Fujita \cite{Fujita} considered corresponding semilinear parabolic equation 
\begin{equation}\label{fu}
\left\{\begin{array}{lll}
 u_t-\Delta u=u^p, & \text { in } \mathbb{R}^N \times(0, \infty) ,\\{}\\
u(x, 0)=u_0(x) \geq 0,& \text { in } \mathbb{R}^N,
\end{array}\right.
\end{equation}
and proved that:\\
  \setword{(i)}{i}  If $1<p<1+\frac{2}{N}:=p_F$, there is no global solution to the problem \eqref{fu}. \\
    (ii) If $p>p_F$ and for a sufficiently small $u_0(x)$, there exists a global positive solution to the problem \eqref{fu}.
    
    The number $p_F$ came to be referred to as the Fujita critical exponents. Moreover, the critical case  $p=p_F$ was considered by Hayakawa \cite{Hayakawa} for $N=1,2$, Sugitani \cite{SG} and Kobayashi et al. \cite{koba} for $N\geq3$. They have been established that the $p=p_F$  also belongs to case \ref{i}. Various proofs for blow-up in the critical case have been demonstrated by \cite{aronson,Pinsky,Weissler}.  To look a widely about the background of blow up phenome refer to review papers \cite{review2,review}.

    Later, Lee and Ni \cite{t2} studied the large-time behavior of solutions to the problem \eqref{fu} for $p>p_F$ with $u_0\geq 0$ and $$u_0(x)\sim |x|^{-\omega}\,\,\, \text{at}\,\,\,|x|\rightarrow\infty.$$ They provided that for $\omega<\frac{2}{p-1}$ the solution is not global in time, and if $u_0$ is sufficiently small with $\omega\geq\frac{2}{p-1}$, then there exists a global solution to the problem \eqref{fu}. The number $\omega_{Crit}=\frac{2}{p-1}$ has been designated as the second critical exponent.

In \cite{Bandleandlevine}, Bandle and Levine observed the semilinear heat equation under Dirichlet boundary conditions in exterior domains: 
\begin{equation}\label{bandle}
\left\{\begin{array}{lll}
 u_t-\Delta u=u^p, & \text { in } D^c \times(0, \infty),\\{}\\
u(x, 0)=u_0(x),& \text { in } D^c,\\{}\\
 u(x, t)  =0 , &  \text { in } \partial D\times (0, \infty),
\end{array}\right.
\end{equation}
\noindent where they established that the problem  \eqref{bandle} possesses no global positive solutions if $1<p<p_F$ and there exist global solutions if $p>p_F$. Moreover, critical case $p=p_F$ was shown by Suzuki \cite{suziki} for $N\geq3$ and by Ikeda and Sobajima \cite{Ikeda1} for $N=2$. There are several extensions of exterior Dirichlet boundary conditions that have to be mentioned \cite{e5,jlelis,sun,sun2}.

Then, Levine and Zhang \cite{Levinezhang} investigated the problem \eqref{fu} under Neumann boundary condition:
\begin{equation}\label{levineZ}
\frac{\partial u(x, t)}{\partial n}=0, \text { in } \partial D\times (0, \infty).
\end{equation}
They showed that the critical exponent remains in the blow-up case if $1<p\leq p_F$ and there exists a global solution to the problem \eqref{levineZ} when $p>p_F$. For more results about  Neumann boundary conditions in exterior domains, refer to, for instance, \cite{Guo,Zhang}.

In addition, Rault \cite{Rault} considered the problem \eqref{fu} under Robin boundary conditions:
\begin{equation}
\frac{\partial u(x, t)}{\partial n}+\alpha u(x,t)=0, \text { in } \partial D\times (0, \infty),
\end{equation}
using a truncation argument and a comparison principle, he established that the critical exponent is the same as $p_F$. For further results related to Robin boundary conditions in exterior domains, see e.g.  \cite{3,Ikeda}.

Bandle, Levine, and Zhang considered the following problem in \cite{BLZ}:
\begin{equation}\label{aaa}
\left\{\begin{array}{lll}
u_t-\Delta u=|u|^p+f(x), & \text { in } 	D^c \times(0, \infty) , \\{}\\
u(x, 0)=u_0(x), & \text { in } 	D^c,\\{}\\
u(x, t)=0, & \text { in } \partial D.
\end{array}\right.
\end{equation}
They obtained the following outcomes:\\
(i) If $p< \frac{N}{N-2}$ and $\int_{D^c} f(x) d x>0$, then Problem \eqref{aaa} has no global solutions.\\
(ii) If $p>\frac{N}{N-2}$ and $N\geq3$, then for sufficiently small $f(x),u_0(x)$  there exists  global solutions to the problem \eqref{aaa}.\\
(iii) For any $p>1$, there exists a negative $f(x)$ such that equation \eqref{aaa}  admits a global solution. 

As can be seen from above, the forcing term $f(x)$ affects the critical exponent, and in consequence, it will jump from $1+\frac{2}{N}$ to $\frac{N}{N-2}$.

In \cite{Gala}, Galaktionov and Pohozaev generalized the problem \eqref{fu} as follows:
\begin{equation}\label{qqq}
\left\{\begin{array}{lll}
 u_t+(-\Delta)^m u=|u|^p,& \text { in } 	\R^N \times(0, \infty) , \\{}\\
 u(x, 0)=u_0(x), & \text { in }	\R^N,
\end{array}\right.
\end{equation}
where $m > 1$, with bounded integrable initial data. They showed that  $1<p\leq 1+\frac{2m}{N}$ blows up in finite time while $p>1+\frac{2m}{N}$ then there exists a global solution to the problem \eqref{qqq} for sufficiently small initial data. Then, Caristi and Mitidieri \cite{Caristi}, showed that there exists a global solution to the problem \eqref{qqq} for slow decay
initial data $0 \leqslant u_0(x) \leqslant \frac{C}{1+|x|^\beta},$ with $\beta>\frac{2 m}{p-1}$ and $p>1+\frac{4}{n}$. Furthermore, Gazzola and Grunau \cite{gazolla}  extend the existence of a global solution with nonlinearity term $|u|^{p-1}u$ for $m=2$ and $\beta=\frac{2 m}{p-1}$. Some other related results in this direction, see e.g.  \cite{Fer,gaz,gaz2}. Later,  Majdoub \cite{majdoub} analyzed the problem \eqref{qqq} with forcing term $f(x)$, where he stated that the critical exponent is $\frac{N}{N-2m}$.

It is important to highlight that numerous papers  have investigated several extensions of the nonexistence of the higher-order semilinear parabolic equations in $\mathbb{R}^N$ \cite{Egorov1,Gala,Philippin} and evolution equations with higher-order time-derivatives \cite{Kirane,laptevgg,e7} in exterior domains. 

To the best of our knowledge, critical exponents in the sense of Fujita have not yet been studied for high-order parabolic problems on exterior domains. This is probably due to various technical difficulties arising from boundary conditions. This fact motivates us more to study the critical behavior of high-order semilinear parabolic exterior problems, one of which is problem \eqref{main}. We hope that our results contribute to the understanding of the blow-up phenomenon in semilinear heat equations with biharmonic operators and provide valuable information for future research. 
\subsection{Main results}
Before stating the main results, we first have to define the weak solutions to the problems I-VI. There are six boundary conditions, however, by analyzing behavior, we can divide them into three groups. Problems were examined based on the following assumptions: $p>1, u_0 \in L_{l o c}^1\left(\overline{D^c}\right)$,  and   $f \not \equiv 0$. 

The existence of a local solution to problems I-VI follows from \cite{pazy}, where the existence of a unique local classical solution to abstract parabolic problems including problems I-VI was studied. Therefore, we shall not revisit the question of local existence and instead concentrate on the global behavior of solutions to problems I-VI.
  \begin{Def}\label{def1}
      We say that $u$ is a  weak solution to  problems   I-VI if there exists $0<T<\infty$ such that $u \in L_{l o c}^p\left(\overline{ Q}\right)$, and satisfies
      \end{Def}
\begin{equation*}
    \begin{aligned}
        \int_Q  |u|^{p}\varphi_j\mathrm{d}x \mathrm{d}t +\int_Q f(x) \varphi_j(x,t)\mathrm{d}x \mathrm{d}t +&\int_{D^c} u_0(x)\varphi_j(x,0)\mathrm{d}x \\
        &= -\int_Q u\partial_t\varphi_j\mathrm{d}x \mathrm{d}t +\int_Q u\Delta^2 \varphi_j\mathrm{d}x \mathrm{d}t,
    \end{aligned}
\end{equation*}
where  $Q:=[0, T] \times D^c,$ and for every function $\varphi_j \in C_{t, x}^{1,4}\left(\overline{ Q}\right) \cap C\left(Q\right)$, $\varphi \geq 0$  such that \\
(i) $\varphi_j(t, \cdot) \equiv 0, t \geq T $;\\       
(ii) $\varphi_j(\cdot, x) \equiv 0,|x| \geq R \text { here } R>1$;\\
\noindent    for $j\in\{1,2,3,4,5,6\}$. The property (i) and property (ii) are the same for all test functions and the property (iii) will identify the difference between the weak solutions to problems I-VI.
\begin{center}
\begin{tabular}{ p{5.5em}  r  r  r r  } 
\hline
 \multicolumn{4}{c}{ (iii)} \\
 \hline
 \rule{0pt}{4ex} 
Problem I & $\varphi_1=\Delta \varphi_1=0$ on $\partial D$; & Problem IV & $\frac{\partial\varphi_4}{\partial n}=\frac{\partial\Delta\varphi_4}{\partial n}=0$ on $\partial D$;  \\ 
\rule{0pt}{4ex} 
Problem II & $\varphi_{2}=\frac{\partial\varphi_2}{\partial n} =0$ on $\partial D;$ & Problem V & $\Delta\varphi_5=\frac{\partial\Delta\varphi_5}{\partial n}=0$ on $\partial D$;  \\ 
\rule{0pt}{4ex} 
Problem III & $\frac{\partial\varphi_3}{\partial n}=\Delta\varphi_3=0$ on $\partial D$; & Problem VI &  $\varphi_{6}=\frac{\partial\Delta\varphi_6}{\partial n} =0$ on $\partial D$. \\ 

\end{tabular}
\end{center}

\vspace{1mm}
 Furthermore, if $T>0$ is allowed to be chosen arbitrarily, then $u$ is referred to as a global weak solution to the problem.

Below we formulate our first result.
\begin{Th}\label{th.n=3} Suppose that $u_0 \in L_{l o c}^1(D^c)$, and $f \in L_{l o c}^1(D^c)\cap L^\infty( D^c)$, $$\int_{D^c}A(x)f(x)dx>0,$$ where
 \begin{equation}\label{a(x).gen}
A(x)= \begin{cases} \text{harmonic function \eqref{A(x)},} & \text {Problem I and VI, } \\
\text{biharmonic function \eqref{H(x)},} & \text {Problem II, } \\ 1, & \text { Problem III-V.}\\
\end{cases}
\end{equation} 
Then, we have 
    \begin{itemize}
        \item[(i)] If $N=2,3,4$, then for all $p>1$, problems I-VI admit no global (in time) weak solution.
        \item[(ii)] Let $N\geq 5$  and  $1<p\leq p_{Crit}:=\frac{N}{N-4}$, then problems I-VI admit no global (in time) weak solution.
        \item[(iii)] When $N\geq 5$ and $p>p_{Crit}$, problems I-VI  admit global (stationary) solutions for some sufficiently small $f>0$ and $u_0>0$.
    \end{itemize}
\end{Th}
Let us consider the problems I-VI for the following semilinear biharmonic Poisson equation
\begin{equation}\label{Ellip}
\Delta^2 u(x)=|u(x)|^{p}+f(x),  \text { in }  D^c.
\end{equation}
Then the following statement is true.
\begin{Cor}\label{Cor1}
Assume that $f \in L_{l o c}^1(D^c)\cap L^\infty( D^c)$ and $f\geq 0$. Then
\begin{itemize}
\item[(i)] If $N=2,3,4$, then for all $p>1$, problems I-VI for the equation \eqref{Ellip} admits no weak solutions.
\item[(ii)] Let $N\geq 5$ and $1<p\leq p_{Crit}:=\frac{N}{N-4}$, then problems I-VI for the equation \eqref{Ellip} admits no weak solutions. 
\item[(iii)] When $N\geq 5$ and $p>p_{Crit}$, problems I-VI for the equation \eqref{Ellip} admit positive solutions.
    \end{itemize}
\end{Cor}

\begin{Rem}
Analyzing Theorem \ref{th.n=3}, we can draw the following conclusions:
\begin{itemize}
\item[(a)] The critical exponent $$p_{Crit}:=\left\{\begin{array}{cc}
\infty, & N=1,2,3,4, \\{}\\
\frac{N}{N-4},  & N\geq 5\end{array}\right.$$ for exterior problems I-VI coincide with the critical exponent of problem \eqref{main} when $D^c=\mathbb{R}^N$, which was obtained by Majdoub in \cite{majdoub}. Also comparing Theorem \ref{th.n=3} and Corollary \ref{Cor1}, it is easy to notice that parabolic equation \eqref{main} and biharmonic Poisson equation \eqref{Ellip} have the same critical exponents $p_{Crit}$. It is easy to see that function $f(x)$ does not affect the critical exponent of the biharmonic Poisson equation, and the results of Corollary \ref{Cor1} remain true even in case $f(x)=0$.
\item[(b)] Part (iii) of Theorem \ref{th.n=3} states the global existence only for stationary positive solutions. Unfortunately, at the moment we do not know any information about global non-stationary (positive or sign-changing) solutions. This is due to technical difficulties associated with the study of Green's functions of problems I-VI, which require separate studies. For now these questions are left open. 
\item[(c)] Of particular interest is the study of critical behaviors in case $$\int_{D^c}f(x)A(x)dx=0,\,f\not\equiv 0,$$ which is not included in Theorem \ref{th.n=3}. The approaches we use do not allow us to study this case, therefore this question also remains open. As far as we know, a similar question was posed by Bandle et al. in \cite{Bandleandlevine} for the classical heat equation $u_t-\Delta u=|u|^p+f(x),\,x\in \mathbb{R}^N$, and still remains open.
\item[(d)] Unlike case $D^c\equiv \mathbb{R}^N$, where assumption $\int\limits_{\mathbb{R}^N}f(x)dx>0$ is imposed on $f$, in the case of an exterior problem it is assumed that $\int\limits_{D^c}f(x)A(x)dx>0$. Here $A(x)$ is given by \eqref{a(x).gen}.
\end{itemize}
\end{Rem}

In part (iii) of Theorem \ref{th.n=3}, it is said that for $p>p_{Crit},$ a global solution exists only for sufficiently small $u_0$ and $f$. Therefore, it makes sense to consider this case separately and study the second critical exponent in the sense of Lee-Ni, which we formulate below. Before, for $\omega<N$ we introduce the following sets
$$
\mathcal{I}_\omega^{+}=\left\{f \in C\left(D^c\right): f \geq 0 \quad \text{and} \quad f(x) \geq C|x|^{-\omega} \text { for sufficiently large }|x|\right\},
$$
and
$$
\mathcal{I}_\omega^{-}=\left\{f \in C\left(D^c\right): f>0 \quad \text{and} \quad f(x) \leq C|x|^{-\omega} \text { for sufficiently large }|x|\right\},
$$
where $\omega<N$ and $C>0$ is an arbitrary constant.

Now it is time to formulate the second main result.
\begin{Th}[Second critical exponents]\label{th.2}
    Suppose that $u_0 \in L_{l o c}^1(D^c)$, and $f \in C( D^c),  f \not \equiv 0$, and $p>p_{Crit}$. 
    \begin{itemize}
        \item[(i)] Let $f \in \mathcal{I}_\omega^{+}$ and $\omega<0$, then problems I-VI do not admit global positive weak solutions;\\
        \item[(ii)] \text {Assume that} $f \in \mathcal{I}_\omega^{+}$ and
        $$0\leq \omega< \omega_{Crit}:=\frac{4p}{p-1},$$
        then problems I-VI possess no global positive weak solutions;\\
        \item[(iii)] If $$\omega_{Crit}\leq \omega<N,$$
        then problems I-VI admit positive global (stationary) classical solutions for some $f \in \mathcal{I}_\omega^{-}$ and for some $u_0>0$.
     \end{itemize}
\end{Th}
\begin{Rem}
Since $p > p_{Crit}:=\frac{N}{N-4},$ we have that the set of $\omega$ satisfying $\frac{4p}{p-1}\leq \omega<N$ is nonempty. The second critical exponent or critical exponent in the sense of Lee-Ni $\omega_{Crit}:=\frac{4p}{p-1}$ of problems I-VI depends on the forcing term $f$, but not on the initial data $u_0$.
\end{Rem}

Let us now study the equation \eqref{main} without forcing term $f(x)$, that is,
\begin{equation}\label{main1}
\left\{\begin{array}{lll}
u_t+\Delta^2 u=|u|^{p}, & \text { in }  D^c \times(0, \infty) , \\{}\\
u(x, 0)=u_0(x), & \text { in } D^c. 
\end{array}\right.
\end{equation}
First, we give a definition of a weak solution to problems I-VI for equation \eqref{main1}.
  \begin{Def}\label{def2}
      We say that $u$ is a  weak solution to  problems   I-VI if there exists $0<T<\infty$ such that $u \in L_{l o c}^p\left(\overline{ Q}\right)$, and satisfies
\begin{equation*}
    \begin{aligned}
        \int_Q  |u|^{p}\varphi_j\mathrm{d}x \mathrm{d}t +&\int_{D^c} u_0(x)\varphi_j(x,0)\mathrm{d}x = -\int_Q u\partial_t\varphi_j\mathrm{d}x \mathrm{d}t +\int_Q u\Delta^2 \varphi_j\mathrm{d}x \mathrm{d}t,
    \end{aligned}
\end{equation*}
 for every function $\varphi_j \in C_{t, x}^{1,4}\left(\overline{ Q}\right) \cap C\left(Q\right)$, $\varphi \geq 0,$ which has the same properties as in Definition \ref{def1}.\end{Def}
Then we get the following nonexistence result.
\begin{Th}\label{th.f}
Assume that $A(x)$ is as \eqref{a(x).gen} and  $u_0 \in L_{l o c}^1(D^c)$. If $\int_{D^c}u_0(x)A(x)dx\geq 0,\,\,u_0\not\equiv 0,$ then problems I-VI for \eqref{main1} admit no global weak solution while $$1<p\leq p_{Fuj}:=1+\frac{4}{N}.$$
\end{Th}

\begin{Rem}
 Theorem \ref{th.2} assumes that $\int_{D^c}u_0(x)A(x)dx\geq 0,$ $u_0\not\equiv 0$. This means that the initial data $u_0$ that satisfy $\int_{D^c}u_0(x)A(x)dx=0,$ but $u_0\not\equiv 0$ is also included in this class of functions.
\end{Rem}
\begin{Rem} Theorem \ref{th.f} does not contain results for the supercritical case $p>1+\frac{4}{N}$. However, the number $p_{Fuj}:=1+\frac{4}{N}$ coincides with the Fujita critical exponent for the biharmonic heat equation (see \cite{Gala, Caristi, gazolla, Egorov1}) $$u_t+\Delta^2 u=|u|^{p}$$ given in the whole space $\mathbb{R}^N$. Therefore, we conjecture that, if $$p>p_{Fuj}=1+\frac{4}{N},$$ then the problems I-VI admit global in-time solutions for some sufficiently small $u_0$. The proof of case $p>1+\frac{4}{N}$ is based on the study of the properties of Green's functions of problems I-VI, which requires separate consideration.
\end{Rem}

By the standard argument, we can verify that there exists a unique local-in-time nonnegative solution to the \eqref{main1}.  Therefore, let us define an upper bound for the lifespan of a local solution $u$ as 
$$
\begin{aligned}
    T_\varepsilon=\text{LifeSpan} (u_\varepsilon) :=\sup _{T>0} &\{ u \text { is a weak solution to \eqref{main1}  with I-VI boundary}\\
&\text {  conditions in }[0, T) \times D^c\} \text {. }
\end{aligned}
$$
Consequently, we can present the following assertion.
\begin{Th}\label{lf}
    Let $u_\varepsilon$ is the unique local weak solution to \eqref{main1}  with I-VI boundary conditions. Assume that $1<p \leq p_{\text {Fuj }}$, $N\geq 3$ ,  and $u_0 \in L_{l o c}^1(D^c)$, $$\int_{D^c}u_0(x)A(x)dx\geq 0,\,\,u_0\not\equiv 0,$$ where $A(x)$ is given by \eqref{a(x).gen}. Then, there exists $\varepsilon_0>0$ such that for any $\varepsilon \in\left(0, \varepsilon_0\right]$ it holds

$$
T_\varepsilon  \leq \begin{cases}C \varepsilon^{-\left(\frac{1}{p-1}-\frac{N}{4}\right)^{-1}} & \text { if } p \in\left(1, p_{\text {Fuj }}\right), \\{}\\ \exp \left(C \varepsilon^{-(p-1)}\right) & \text { if } p=p_{\text {Fuj }}.\end{cases}
$$
\end{Th}
\begin{Rem}
At the moment we do not know the lifespan estimates for local solutions in the one and two-dimensional cases. This is because the approach we use is not suitable for these cases. 
\end{Rem}

\section{Some classes of test functions and associated useful estimates}

Before diving into the main topic, there are a few things we need to discuss first. Especially, we are interested in the properties of the test functions. In each subsection, we define a test function depending on the problem's type. Then we provide a convenient estimate of the test functions,  which will be advantageous in a proof.

\subsection{Properties of $\boldsymbol{\varphi_1(x,t)}$}Let $H(x)$ be a harmonic function in $D^c$ given by
\begin{equation}\label{A(x)}
H(x)= \begin{cases}\ln |x|, & \text { if } \quad N=2, \\ 1-|x|^{2-N}, & \text { if } \quad N \geq 3.\end{cases}
\end{equation}

\begin{Lem}\label{a.n2.estimate}
The function $H(x)$  satisfies the following properties:\\
(i) $H(x) \geq 0$;\\
(ii) $\Delta H(x)=\Delta^2 H(x)=0$ in $D^c$;\\
(iii) $H(x)=\Delta H(x)=\frac{\partial\left( \Delta H(x)\right)}{\partial \nu}=0$ on $\partial D$.
\end{Lem}
\begin{proof}
    The first property (i) is clear from the \eqref{A(x)}.  The second property (ii) implies immediately from the harmonic function $$\Delta H(x)=\Delta^2 H(x)=0.$$  Furthermore, the third property (iii) is obvious because of $H(x)=0$ on $\partial D$.
\end{proof}
\noindent  Test functions are defined as follows:
\begin{equation}\label{phi}
\varphi_1(t, x)=\eta(t) \psi(x), \quad(t, x) \in  Q.    
\end{equation}

\noindent For sufficiently large $T$

\begin{equation}\label{zeta}
    \eta(t)=\zeta^{\ell}\left(\frac{t}{T}\right), \quad t>0,
\end{equation}
\noindent where $\ell\geq\frac{4p}{p-1}$, and $\zeta \in C^{\infty}(\mathbb{R})$ is a cut-off function satisfying
\begin{equation}\label{zeta.s}
   \zeta(s)= \begin{cases}1, & \text { if } 0\leq s \leq 
   \frac{1}{2}, \\ \searrow, & \text { if } \frac{1}{2}<s<1, \\ 0, & \text { if } s \geq 1,\end{cases} 
\end{equation}
\noindent and for sufficiently large $R$ 

\begin{equation}\label{phi1}
   \psi(x)=H(x) \xi^{\ell}\left(\frac{|x|}{R}\right), \quad x \in D^c,
\end{equation}

\noindent where $H(x)$ is harmonic function \eqref{A(x)}, and $\xi: \mathbb{R} \rightarrow[0,1]$ is a smooth function 
\begin{equation}\label{xi.s}
   \xi(s)= \begin{cases}1, & \text { if } 0\leq s \leq 1, \\ \searrow, & \text { if } 1<s<2, \\ 0, & \text { if } s \geq 2.\end{cases} 
\end{equation}
In addition, we require the test function to satisfy the following properties:
$$|\xi'(s)|\leq C, \quad |\xi''(s)|\leq C$$ and $$\xi'''(s)|\leq C ,\quad |\xi^{(IV)}(s)|\leq C. $$

Moreover, the harmonic function \eqref{A(x)} implies to consider two cases separately, when $N=2$ and when $N\geq3$. 

\subsection{The case N = 2}
\begin{Lem}\label{A(x.estimate).N2}
     For $R<|x|<2 R$, the following estimate holds
$$
 H(x) \leq C \ln R,$$ and $$|\nabla H(x)| \leq C R^{-1}.
$$
\end{Lem}
\begin{proof} The harmonic function \eqref{A(x)} can be estimated as follows:
\begin{align*}
H(x)&= \ln |x|\leq  \ln 2R=  \ln{R} +\ln 2 \\&\leq  \ln R+\ln R\leq C\ln R,\end{align*}
from direct computation, we get
$$ \nabla H(x)=\frac{x}{|x|^2},\quad |\nabla H(x)|\leq CR^{-1}.$$
which verifies the proof.
\end{proof}

\begin{Lem}\label{estimate.t.n=2} 
    For $R>0$ and $T>0$, the subsequent estimate is valid:
    \begin{equation*}
         \int_Q   \varphi_1^{-\frac{1}{p-1}}(x,t)\left| \partial_t\varphi_1(x,t)\right|^{p^{\prime}} \mathrm{d}x \mathrm{d}t\leq C T^{1-p^{\prime}}R^{2} \ln R.
    \end{equation*}
   
\end{Lem}

\begin{proof}

Employing the attributes of the test functions specified in equation \eqref{phi}, we can derive
\begin{equation*}
      \varphi_1^{-\frac{1}{p-1}}(x,t)\left| \partial_t \varphi_1(x,t)\right|^{p^{\prime}}=  \psi(x)\eta^{-\frac{1}{p-1}}(t)\big|\partial_t\eta(t)\big|^{p^{\prime}},
\end{equation*}
which reformulates the following integral
    $$
\begin{aligned}
 \int_Q  \varphi_1^{-\frac{1}{p-1}}(x,t)\left|\partial_t \varphi_1(x,t)\right|^{p^{\prime}}\mathrm{d}x \mathrm{d}t =\left(\int_{D^c}\psi(x)\mathrm{d}x \right)\left(\int_0^T \eta^{-\frac{1}{p-1}}(t)\big|\partial_t\eta(t)\big|^{p^{\prime}} \mathrm{d}t\right).
\end{aligned}
$$
Further, using the \eqref{phi1} and lemma \ref{A(x.estimate).N2}, we have
 
\begin{equation}\label{N2.xi.estimate}
    \begin{aligned}
        \int_{D^c}\psi(x)\mathrm{d}x&=\int_{1<|x|<2R} H(x) \xi^l\left(\frac{|x|}{R}\right)\mathrm{d}x\\
        & \leq C \ln R\int_{1<|x|<2R} \xi^l\left(\frac{|x|}{R}\right)\mathrm{d}x\\
         & \stackrel{|x|=r}{\leq}  C \ln R \int_{1}^{2R} r \mathrm{d}r\\
         & \leq CR^{2} \ln R.
    \end{aligned}
\end{equation}
On the other hand, the cut-off function \eqref{zeta} implies that 

\begin{equation}\label{N2.derivetavizeta.estimate}
\begin{aligned}
\int_0^T \eta^{-\frac{1}{p-1}}(t) \left| \partial_t\eta(t)\right|^{p^{\prime}} \mathrm{d}t& =T^{-p^{\prime}} \int_{\frac{T}{2}}^T \zeta^{-\frac{l}{p-1}}\left(\frac{t}{T}\right) \zeta^{(l-1) p'}\left(\frac{t}{T}\right)\mathrm{d}t
 \\
& =CT^{-p^{\prime}} \int^T_{\frac{T}{2}} \zeta^{l-p^{\prime}}\left(\frac{t}{T}\right)\mathrm{d}t\\
&=C T^{1-p^{\prime}} \int_{\frac{1}{2}}^1 \zeta^{l-p^{\prime}}(s) \mathrm{d}s\\
& \leq C T^{1-p^{\prime}}.
\end{aligned}
\end{equation}

\noindent Therefore, from \eqref{N2.xi.estimate} and \eqref{N2.derivetavizeta.estimate} we get the desired results.
 \end{proof}

The subsequent estimate arises from the standard calculation. The function $\xi$ is radial, which means the estimate can be directly calculated by changing to the polar coordinates. The power of $R$ provides a relation of the derivatives to the cut-off function. 

\begin{Lem}\label{samet.lemma}
    For $R<|x|<2 R$, the following estimates hold:
$$\left|\nabla\left(\Delta \xi^{\ell}\left(\frac{|x|}{R}\right)\right)\right|  \leq C R^{-3} \xi^{\ell-4}\left(\frac{|x|}{R}\right),$$ and
$$\left|\Delta^2 \xi^{\ell}\left(\frac{|x|}{R}\right)\right|  \leq C R^{-4} \xi^{\ell-4}\left(\frac{|x|}{R}\right).
$$
\end{Lem}

\noindent Furthermore, by \eqref{phi1}, we have
\begin{equation}\label{bxi}
\begin{aligned}
\Delta^2\psi(x)
= & \xi^{\ell}\left(\frac{|x|}{R}\right) \Delta^2 H(x)+4 \nabla(\Delta H(x)) \cdot \nabla \xi^{\ell}\left(\frac{|x|}{R}\right)\\&+6 \Delta H(x) \Delta \xi^{\ell}\left(\frac{|x|}{R}\right) +4 \nabla H(x) \cdot \nabla\left(\Delta \xi^{\ell}\left(\frac{|x|}{R}\right)\right)\\&+H(x) \Delta^2 \xi^{\ell}\left(\frac{|x|}{R}\right).
\end{aligned}
\end{equation}

\noindent Taking into account \eqref{A(x)}, we deduce 
$$
\begin{aligned}
\Delta^2\psi(x)= & H(x) \Delta^2 \xi^{\ell}\left(\frac{|x|}{R}\right)+4 \nabla H(x) \cdot \nabla\left(\Delta \xi^{\ell}\left(\frac{|x|}{R}\right)\right),
\end{aligned}
$$

\noindent and from the lemma \ref{A(x.estimate).N2} and lemma \ref{samet.lemma}
\begin{equation}\label{N2.biha.xi}
 \begin{aligned}
\left|\Delta^2\psi(x)\right| & \leq C\left(R^{-4} \ln R+ R^{-4}\right) \xi^{\ell-4}\left(\frac{|x|}{R}\right) \\
& \leq C R^{-4} \ln R \xi^{\ell-4}\left(\frac{|x|}{R}\right), \quad R<|x|<2 R .
\end{aligned}   
\end{equation}

\begin{Lem}\label{estimate.b.n=2}
   For sufficiently large $R, T$ the following integral can be evaluated as
    \begin{equation*}
        \int_Q   \varphi_1^{-\frac{1}{p-1}}(x,t)\left|\Delta^2 \varphi_1(x,t)\right|^{p'} \mathrm{d}x \mathrm{d}t \leq CT^1 R^{-\frac{2(p+1)}{p-1}}(\ln{R})^{\frac{p}{p-1}}.
    \end{equation*}   
\end{Lem}

\begin{proof}  Using \eqref{phi}, we deduce 

\begin{equation*}
    \varphi_1^{-\frac{1}{p-1}}(x,t)\left|\Delta^2 \varphi_1(x,t)\right|^{p'}=  \eta(t)\psi(x)^{-\frac{1}{p-1}}(x)\left|\Delta^2\psi(x)\right|^{p^{\prime}},
\end{equation*}
which implies that
     $$
\begin{aligned}
\int_Q \varphi_1^{-\frac{1}{p-1}}(x,t)\left|\Delta^2 \varphi_1(x,t)\right|^{p'} \mathrm{d}x \mathrm{d}t  =\left(\int_0^T\eta(t) \mathrm{d}t \right)\left(\int_{D^c}\psi(x)^{-\frac{1}{p-1}}(x)\left|\Delta^2\psi(x)\right|^{p^{\prime}}\mathrm{d}x \right).
\end{aligned}
$$
    By the definition of the cut-off function \eqref{zeta}, we obtain
\begin{equation}\label{zeta.estimate}
\begin{aligned}
\int_0^T \eta(t) \mathrm{d}t  =\int_0^T \zeta^{\ell}\left(\frac{t}{T}\right) \mathrm{d}t =CT \int_0^1 \zeta^{\ell}(s) \mathrm{d}s= CT.
\end{aligned}
\end{equation}

\noindent Moreover,  lemma \ref{A(x.estimate).N2} and inequality \eqref{N2.biha.xi} imply that

\begin{equation}\label{le2}
\begin{aligned}
 \int_{D^c}\psi(x)^{-\frac{1}{p-1}}(x)\left|\Delta^2\psi(x)\right|^{p^{\prime}}\mathrm{d}x=&\int_{R<|x|<2 R} \left[H(x) \xi^l\left(\frac{|x|}{R}\right)\right]^{-\frac{1}{p-1}} \left|\Delta^2\psi(x)\right|^{p^{\prime}}\mathrm{d}x \\
\leq & C R^{-4 p^{\prime}} (\ln{R})^{\frac{p}{p-1}} \int_{R<|x|<2 R} \xi^{l-4 p^{\prime}}\left(\frac{|x|}{R}\right) \mathrm{d}x\\
\leq&         C R^{-\frac{2(p+1)}{p-1}} (\ln{R})^{\frac{p}{p-1}}.
\end{aligned}
\end{equation}

\noindent To obtain the expected outcome, it is necessary to merge \eqref{zeta.estimate} and \eqref{le2}.
\end{proof}

\subsection{The case when $\boldsymbol{N\geq3}$}

\begin{Lem}\label{A(x.estimate).N=3}
     For $R<|x|<2 R$, we have
$$ H(x) \leq C,$$ and $$|\nabla H(x)| \leq C R^{1-N}.
$$
\end{Lem}
\begin{proof}
From the direct calculation of \eqref{A(x)}, we deduce
    $$
H(x)=\left(1-|x|^{2-N}\right)< \left(1-(2R)^{2-N}\right)\leq C,
$$
also for the second inequality
$$  \nabla H(x)=(N-2)\frac{x}{|x|^{N}}, \quad  |\nabla H(x)|\leq CR^{1-N}.$$
This finalizes our argument.
\end{proof}

\begin{Lem}\label{estimate.t} 
    For $R>0$ and $T>0$ the following estimate holds:
    \begin{equation*}\label{le23}
         \int_Q \varphi_1(x,t)^{-\frac{1}{p-1}}\left| \partial_t \varphi_1(x,t)\right|^{p^{\prime}} \mathrm{d}x \mathrm{d}t\leq C T^{1-p^{\prime}}R^{N} .
    \end{equation*}
   
\end{Lem}

\begin{proof}
By leveraging the properties of the test functions as delineated in equations \eqref{phi} through \eqref{phi1}, we arrive at 
    $$
\begin{aligned}
\int_Q  \varphi_1^{-\frac{1}{p-1}}(x,t)\left| \partial_t \varphi_1(x,t)\right|^{p^{\prime}}\mathrm{d}x \mathrm{d}t =\left(\int_{D^c}\psi(x)\mathrm{d}x \right)\left(\int_0^T \eta^{-\frac{1}{p-1}}(t)\big|\partial_t\eta(t)\big|^{p^{\prime}} \mathrm{d}t\right).
\end{aligned}
$$

\noindent Taking into account definition of test functions \eqref{phi1} and the lemma \ref{A(x.estimate).N=3}, we derive

\begin{equation}\label{N3.xi.estimate}
    \begin{aligned}
        \int_{D^c}\psi(x)\mathrm{d}x&=\int_{1<|x|<2R} H(x) \xi^l\left(\frac{|x|}{R}\right)\mathrm{d}x\\
        & \leq C \int_{1<|x|<2R} \xi^l\left(\frac{|x|}{R}\right)\mathrm{d}x\\
         & \stackrel{|x|=r}{\leq}  C  \int_{1}^{2R} r^{N-1} \mathrm{d}r\\
         & \leq CR^{N} .
    \end{aligned}
\end{equation}
\noindent The integral involving a cut-off function evaluates as \eqref{N2.derivetavizeta.estimate} and combining with \eqref{N3.xi.estimate}, we have the desired result. This finalizes the proof.
\end{proof}

\noindent The following computation is straightforward from \eqref{bxi}
$$
\begin{aligned}
\Delta^2\psi(x)= & H(x) \Delta^2 \xi^{\ell}\left(\frac{|x|}{R}\right)+4 \nabla H(x) \cdot \nabla\left(\Delta \xi^{\ell}\left(\frac{|x|}{R}\right)\right),
\end{aligned}
$$

\noindent and from  lemma \ref{samet.lemma} and lemma \ref{A(x.estimate).N=3}, we derive
\begin{equation}\label{es.n3}
   \begin{aligned}
\left|\Delta^2\psi(x)\right| & \leq C\left(R^{-4} +R^{1-N} R^{-3}\right) \xi^{\ell-4}\left(\frac{|x|}{R}\right) \\
& \leq C R^{-4}  \xi^{\ell-4}\left(\frac{|x|}{R}\right), \quad R<|x|<2 R .
\end{aligned} 
\end{equation}

\begin{Lem}\label{estimate.b.n=3}
    The following integral can be evaluated as
    \begin{equation*}
        \int_Q  \varphi_1^{-\frac{1}{p-1}}(x,t)\left|\Delta^2 \varphi_1(x,t)\right|^{p'} \mathrm{d}x \mathrm{d}t \leq CT^1 R^{N-4p'}.
    \end{equation*}   
\end{Lem}

\begin{proof}  By \eqref{phi}, we  obtain following 
    $$
\begin{aligned}
\int_Q  \varphi_1^{-\frac{1}{p-1}}(x,t)\left|\Delta^2 \varphi_1(x,t)\right|^{p'}\mathrm{d}x \mathrm{d}t =\left(\int_0^T\eta(t) \mathrm{d}t \right)\left(\int_{D^c}\psi^{-\frac{1}{p-1}}(x)\left|\Delta^2\psi(x)\right|^{p^{\prime}}\mathrm{d}x \right).
\end{aligned}
$$

\noindent From \eqref{zeta.estimate}, we have
   $$
\begin{aligned}
\int_Q   \varphi_1^{-\frac{1}{p-1}}(x,t)\left|\Delta^2 \varphi_1(x,t)\right|^{p'}\mathrm{d}x \mathrm{d}t = CT\int_{D^c}\psi^{-\frac{1}{p-1}}(x)\left|\Delta^2\psi(x)\right|^{p^{\prime}}\mathrm{d}x,
\end{aligned}
$$

\noindent By using the lemma \eqref{A(x.estimate).N=3} and estimate \eqref{es.n3}, we obtain following

\begin{equation*}
\begin{aligned}
 \int_Q   \varphi_1^{-\frac{1}{p-1}}(x,t)\left|\Delta^2 \varphi_1(x,t)\right|^{p'}\mathrm{d}x \mathrm{d}t &\leq  CT R^{-4 p^{\prime}}  \int_{R<|x|<2 R} \xi^{l-4 p^{\prime}}\left(\frac{|x|}{R}\right) \mathrm{d}x\\
&\stackrel{|x|=r}{\leq} CT R^{-4 p^{\prime}}  \int_{R<r<2R}r^{N-1}\mathrm{d}r\\
&\leq         C TR^{N-4 p^{\prime}} ,
\end{aligned}
\end{equation*}

\noindent which completes the proof.
\end{proof}

\subsection{Test functions with logarithmic arguments}

We introduce the functions of the form
\begin{equation}\label{phic}
   \varphi_1(t, x)=\eta(t) \psi(x)=\zeta^{\ell}\left(\frac{t}{T}\right) H(x) \mathcal{F}^{\ell}\left(\frac{\ln \left( \frac{|x|}{\sqrt{R}}\right)}{\ln (\sqrt{R})}\right), \quad(t, x) \in Q,
\end{equation}

\noindent where $H(x)$ is harmonic function \eqref{A(x)} , $\eta(t)$ is cut-off function as \eqref{zeta} and  $\mathcal{F}: \mathbb{R} \rightarrow[0,1]$ is a smooth function
\begin{equation}\label{f.s}
\mathcal{F}(s)= \begin{cases}1, & \text { if }-1 \leq s \leq 0, \\ \searrow, & \text { if } 0<s<1, \\ 0, & \text { if } s \geq 1.\end{cases}    
\end{equation}

\begin{Lem}
    For $N\geq 5$ and $p=\frac{N}{N-4}$ the following estimate holds:
    \begin{equation}\label{critpart}
         \int_Q  \varphi_1^{-\frac{1}{p-1}}(x,t)\left|\partial_t \varphi_1(x,t)\right|^{p^{\prime}}\mathrm{d}x \mathrm{d}t \leq C T^{1-\frac{N}{4}}R^{N}.
    \end{equation}
   
\end{Lem}
\begin{proof}
Applying the properties of the test functions as defined in \eqref{phic}, we derive
     $$
\begin{aligned}
 \int_Q  \varphi_1^{-\frac{1}{p-1}}(x,t)\left|\partial_t\varphi_1(x,t)\right|^{\frac{N}{4}}\mathrm{d}x \mathrm{d}t&=\left(\int_{D^c} \psi(x)\mathrm{d}x \right)\left(\int_0^T \eta^{-\frac{1}{p-1}}(t)\big|\partial_t\eta(t)\big|^{\frac{N}{4}} \mathrm{d}t\right) .
\end{aligned}
$$

\noindent implying the lemma \ref{A(x.estimate).N=3}, we get
\begin{equation}
    \begin{aligned}\label{jjj}
        \int_{D^c} \psi(x) \mathrm{d}x &=\int_{1<|x|<R} H(x)  \mathcal{F}^{\ell}\left(\frac{\ln \left( \frac{|x|}{\sqrt{R}}\right)}{\ln (\sqrt{R})}\right)\mathrm{d}x \\
         & \leq C\int_{1<|x|<R} \mathcal{F}^{\ell}\left(\frac{\ln \left( \frac{|x|}{\sqrt{R}}\right)}{\ln (\sqrt{R})}\right)\mathrm{d}x\\
        & \stackrel{|x|=r}{\leq} C\int_1^{R} r^{N-1} \mathrm{d}r\\
        &\leq CR^{N}.
    \end{aligned}
\end{equation}

\noindent Therefore, \eqref{critpart} follows from \eqref{zeta.estimate} and \eqref{jjj}.
\end{proof}

The evidence of the subsequent lemma can be obtained through simple calculations because we deal with the radial function $\mathcal{F}(|x|)$.

\begin{Lem}\label{salistiry}
    The following estimates hold:

$$
\left|\nabla\cdot\Delta \mathcal{F}^{\ell}\left(\frac{\ln\left(\frac{|x|}{\sqrt{R}}\right) }{\ln{\sqrt{R}}}\right)\right|  \leq C  \frac{1}{|x|^3\ln{\sqrt{R}}}\mathcal{F}^{\ell-4}\left(\frac{\ln\left(\frac{|x|}{\sqrt{R}}\right) }{\ln{\sqrt{R}}}\right),$$ and $$\left|\Delta^2 \mathcal{F}^{\ell}\left(\frac{\ln\left(\frac{|x|}{\sqrt{R}}\right) }{\ln{\sqrt{R}}}\right)\right| 
 \leq C  \frac{1}{|x|^4\ln{\sqrt{R}}} \mathcal{F}^{\ell-4}\left(\frac{\ln\left(\frac{|x|}{\sqrt{R}}\right) }{\ln{\sqrt{R}}}\right),$$
where $\cdot$ denotes the inner product in $\mathbb{R}^N$.
\end{Lem}

\noindent From properties of $H(x)$ and from the equation \eqref{bxi} we deduce that
$$
\begin{aligned}
\Delta^2\psi(x)= & H(x) \Delta^2 \mathcal{F}^{\ell}\left(\frac{\ln\left(\frac{|x|}{\sqrt{R}}\right) }{\ln{\sqrt{R}}}\right)+4\nabla H(x)  \nabla\cdot\Delta \mathcal{F}^{\ell}\left(\frac{\ln\left(\frac{|x|}{\sqrt{R}}\right) }{\ln{\sqrt{R}}}\right),
\end{aligned}
$$
here using lemma \ref{salistiry}, we have
\begin{equation}\label{b.estimate}
\begin{aligned}
\left|\Delta^2\psi(x)\right| & \leq C\left(  \frac{1 }{|x|^{4}\ln{\sqrt{R}}}  +  \frac{|x|^{1-N}}{|x|^3\ln{\sqrt{R}}} \right) \mathcal{F}^{\ell-4}\left(\frac{\ln\left(\frac{|x|}{\sqrt{R}}\right) }{\ln{\sqrt{R}}}\right)\\
&\leq C\frac{1 }{|x|^{4}\ln{\sqrt{R}}}  \mathcal{F}^{\ell-4}\left(\frac{\ln\left(\frac{|x|}{\sqrt{R}}\right) }{\ln{\sqrt{R}}}\right).
\end{aligned}
\end{equation}

\begin{Lem}\label{estimate.b}
    The following integral can estimate as
    \begin{equation*}
        \int_Q  \varphi_1^{-\frac{1}{p-1}}(x,t)\left|\Delta^2 \varphi_1(x,t)\right|^{p'} \mathrm{d}x \mathrm{d}t\leq CT^1  (\ln{R})^{-\frac{N}{4}}.
    \end{equation*}
    
\end{Lem}

\begin{proof}
Utilizing the properties of the test functions described in equations \eqref{phi} to \eqref{phi1}, we can infer    
\begin{equation*}
  \int_Q  \varphi_1^{-\frac{1}{p-1}}(x,t)\left|\Delta^2 \varphi_1(x,t)\right|^{p'} \mathrm{d}x \mathrm{d}t=\left(\int_0^T\eta(t) \mathrm{d}t \right)\left(\int_{D^c} \psi_2^{-\frac{1}{p-1}}(x)\left|\Delta^2\psi(x)\right|^{p^{\prime}}\mathrm{d}x \right), 
\end{equation*}

\noindent time variable integral is the same as \eqref{zeta.estimate},  by the inequality \eqref{b.estimate} we have 
$$
\begin{aligned}
    \int_{D^c} &\psi_2^{-\frac{1}{p-1}}(x) \left|\Delta^2\psi(x)\right|^{p'} \mathrm{d}x\\ 
    & =   \int_{1<|x|<R} H(x)^{-\frac{1}{p-1}} \mathcal{F}^{-\frac{\ell}{p-1}}\left(\frac{\ln\left(\frac{|x|}{\sqrt{R}}\right) }{\ln{\sqrt{R}}}\right)\left|\Delta^2\psi(x)\right|^{p'}\mathrm{d}x \\
    &\leq C\int_{1<|x|<R} \mathcal{F}^{-\frac{\ell}{p-1}}\left(\frac{\ln\left(\frac{|x|}{\sqrt{R}}\right) }{\ln{\sqrt{R}}}\right)\left[   \frac{1 }{|x|^{4}\ln{\sqrt{R}}} \mathcal{F}^{\ell-4}\left(\frac{\ln\left(\frac{|x|}{\sqrt{R}}\right) }{\ln{\sqrt{R}}}\right)\right]^{p'}\mathrm{d}x.
\end{aligned}
$$

\noindent Upon substituting variables with polar coordinates and $p=\frac{N}{N-4}$, we obtain
\begin{equation*}
    \begin{aligned}
        \int_{D^c} \psi_2^{-\frac{1}{p-1}}&(x) \left|\Delta^2\psi(x)\right|^{p'} \mathrm{d}x\\
        &\stackrel{|x|=r}{\leq} C(\ln{R})^{-\frac{N}{4}} \int_{\sqrt{R}<r<R} r^{-1}dr+CR^{(1-N)\frac{N}{4}}(\ln{R})^{-\frac{N}{4}}\int_{\sqrt{R}<r<R} r^{\frac{N}{4}-1}dr \\
        &\leq  C\left( (\ln{R})^{-\frac{N}{4}+1}+R^{\frac{N}{2}\left(1-\frac{N}{2}\right)}(\ln{R})^{-\frac{N}{4}}\right).
    \end{aligned}
\end{equation*}
which brings the proof to its conclusion.
\end{proof}

\subsection{Properties of $\boldsymbol{\varphi_2(x, t)}$}
Let $B(x)$ be a biharmonic function in $D^c$ given by

\begin{equation}\label{H(x)}
B(x)= \begin{cases}|x|^2 \ln |x|-|x|^2+\ln |x|+1 & \text { if } \quad N=2, \\ \vspace{1mm} |x|^{-1}+|x|-2 & \text { if } \quad N=3, \\ \vspace{1mm} 2\ln{|x|} +|x|^{-2}- 1 & \text { if } \quad N=4, \\ \vspace{1mm} |x|^{2-N}-1+\frac{N-2}{N-4}\left(1-|x|^{4-N}\right) & \text { if } \quad N \geq 5.\end{cases}
\end{equation}
\begin{Lem}
The function $B(x)$ satisfies the following properties:

    \begin{itemize}
        \item[(i)] $B(x) \geq 0$;
        \item[(ii)]  $\Delta^2 B(x)=0$ in $D^c$;
        \item[(iii)] $B(x)=\frac{\partial B(x)}{\partial n}=\Delta B(x)=0$ on $\partial D$.
    \end{itemize}
\end{Lem}

\noindent We introduce trial functions of the form
\begin{equation}\label{psi1}
  \varphi_2(t, x)=\eta(t) \phi(x)=\zeta^{\ell}\left(\frac{t}{T}\right)B(x) \xi^{\ell}\left(\frac{|x|}{R}\right), \quad(t, x) \in Q.  
\end{equation}

\noindent where $\eta(t)$ cut-off function as \eqref{zeta}, $B(x)$ is a biharmonic function as \eqref{H(x)} and $\xi: \mathbb{R} \rightarrow[0,1]$ is a smooth function as \eqref{xi.s}.

The estimate of the test function will be different depending on the dimension of $\R^N$. However, notice that only case $N=2,4$ differs from other cases. Therefore, we inspect $N=2,4$  and $N\geq 5$ cases of biharmonic function \eqref{H(x)} differently. 

\subsection{The case when $\boldsymbol{N=2}$}
\begin{Lem}\label{H(x.estimate).N2.normal}
     For $R<|x|<2 R$, the following estimates are valid
     \begin{equation*}
         \begin{aligned}
             B(x) &\leq C R^2\ln R,\\
             |\nabla B(x)| &\leq C R\ln{R},\\
             |\Delta B(x)|&\leq C\ln{R},
         \end{aligned}
     \end{equation*}
 and
$$|\nabla\cdot\Delta B(x)|\leq CR^{-1}.$$
\end{Lem}
\begin{proof}
By \eqref{H(x)}, we arrive at
    $$B(x)=|x|^2 \ln |x|-|x|^2+\ln |x|+1\leq C R^2\ln R, $$
by differentiating, we obtain 
$$\nabla B(x)=2x\ln{|x|}-x(1-|x|^{-2}), \quad |\nabla B(x)| \leq CR\ln{R}.$$
Let us define  
    $$
h(r)=r^2 \ln r-r^2+\ln r+1,
$$
then $$
h^{\prime}(r)=2 r \ln r-r+r^{-1},$$ and the second derivative is 
 $$h^{\prime \prime}(r)=2 \ln r+\left(1-r^{-2}\right),$$
thus, we get
$$
\begin{aligned}
&\Delta B(x) =h^{\prime \prime}(r)+\frac{1}{r} h^{\prime}(r)  =4 \ln r,\\& |\Delta B(x)|\leq C \ln R, \quad r=|x| .
\end{aligned}
$$
Then we have
$$\nabla\cdot\Delta B(x)=\frac{x}{|x|^2}, \quad |\nabla\cdot\Delta B(x)|\leq CR^{-1},$$
and that completes the proof.
\end{proof}

\begin{Lem}\label{estimate.t.normal} 
    For $R>0$ and $T>0$ the following estimate holds:
    \begin{equation*}
         \int_Q \varphi_2^{-\frac{1}{p-1}}(x,t)\left|\partial_t \varphi_2(x,t)\right|^{p^{\prime}} \mathrm{d}x \mathrm{d}t\leq C T^{1-p^{\prime}}R^{4}\ln{R}.
    \end{equation*}
   
\end{Lem}

\begin{proof}
Drawing upon the features of the test functions  \eqref{psi1}, we arrive at

    $$
\begin{aligned}
 \int_Q  \varphi_2^{-\frac{1}{p-1}}(x,t)\left|\partial_t \varphi_2(x,t)\right|^{p^{\prime}}\mathrm{d}x \mathrm{d}t =\left(\int_{D^c} \phi(x)\mathrm{d}x \right)\left(\int_0^T \eta^{-\frac{1}{p-1}}(t)\big|\partial_t\eta(t)\big|^{p^{\prime}} \mathrm{d}t\right).
\end{aligned}
$$

\noindent For sufficiently large $R$ and for $l> \frac{4p}{p-1}$, we obtain  
\begin{equation}\label{xi.integral.N2.normal}
    \begin{aligned}
        \int_{D^c} \phi(x)\mathrm{d}x&=\int_{1<|x|<2R} B(x) \xi^l\left(\frac{|x|}{R}\right)\mathrm{d}x\\
        & \leq CR^2\ln{R}\int_{1<|x|<2R} \xi^l\left(\frac{|x|}{R}\right)\mathrm{d}x\\
        & \stackrel{|x|=r}{\leq} CR^2\ln{R}\int_{1<r<2R}  r \mathrm{d}r\\
        &\leq CR^{4}\ln{R}.
    \end{aligned}
\end{equation}

\noindent Note that the lemma \ref{H(x.estimate).N2.normal} was also needed to estimate bigarmonic function in \eqref{xi.integral.N2.normal}. The integral with the cut-off function is provided in \eqref{zeta.estimate}, combining the estimate we complete the proof.
\end{proof}
We have already shown the estimate for the test function in lemma \ref{samet.lemma} and for the biharmonic function in lemma \ref{H(x.estimate).N2.normal}. Therefore, the connection of both lemmas is enough to conclude the following lemma. 

\begin{Lem}\label{samet.lemma.normal.n=2}
    For $R<|x|<2 R$, the following estimates hold:
$$\left|B(x) \Delta^2 \xi^{\ell}\left(\frac{|x|}{R}\right)\right|  \leq C R^{-2}\ln{R}\xi^{\ell-4}\left(\frac{|x|}{R}\right),$$ 
$$\left|\nabla B(x) \cdot \nabla\left(\Delta \xi^{\ell}\left(\frac{|x|}{R}\right)\right)\right|  \leq C R^{-2}\ln{R}\xi^{\ell-4}\left(\frac{|x|}{R}\right),$$ 
$$\left|\Delta B(x) \Delta \xi^{\ell}\left(\frac{|x|}{R}\right)\right| \leq C R^{-2}\ln{R}\xi^{\ell-4}\left(\frac{|x|}{R}\right),$$ and
$$\left|\nabla(\Delta B(x)) \cdot \nabla \xi^{\ell}\left(\frac{|x|}{R}\right)\right|  \leq CR^{-2}\xi^{\ell-4}\left(\frac{|x|}{R}\right),$$
where $\cdot$ denotes the inner product in $\mathbb{R}^N$.
\end{Lem}

\begin{Lem}\label{estimate.b.normal}
    The following integral can be evaluated as
    \begin{equation*}
        \int_Q  \varphi_2^{-\frac{1}{p-1}}(x,t)\left|\Delta^2 \varphi_2(x,t)\right|^{p'} \mathrm{d}x \mathrm{d}t \leq CT^1 R^{-\frac{2}{p-1}} (\ln{R})^{\frac{p}{p-1}} .
    \end{equation*}   
\end{Lem}

\begin{proof}  From test functions can be obtained following 
\begin{equation*}
\begin{aligned}
\int_Q   \varphi_2^{-\frac{1}{p-1}}(x,t)\left|\Delta^2 \varphi_2(x,t)\right|^{p'}\mathrm{d}x \mathrm{d}t=\left(\int_0^T\eta(t) \mathrm{d}t \right)\left(\int_{D^c} \phi^{-\frac{1}{p-1}}(x)\left|\Delta^2\phi(x)\right|^{p^{\prime}}\mathrm{d}x \right).
\end{aligned}    
\end{equation*}

\noindent Moreover, by \eqref{bxi} and lemma \ref{samet.lemma.normal.n=2} it is straightforward to show
\begin{equation}\label{n=2.phi}
    \begin{aligned}
\left|\Delta^2\phi(x)\right| & \leq C\left( R^{-2}\ln{R} +R^{-2}\ln{R} +R^{-2}+ R^{-2}\ln{R}\right) \xi^{\ell-4}\left(\frac{|x|}{R}\right) \\
& \leq  C R^{-2}\ln{R}\xi^{\ell-4}\left(\frac{|x|}{R}\right), \quad R<|x|<2 R .
\end{aligned}
\end{equation}
In addition, by \eqref{n=2.phi} and lemma \ref{H(x.estimate).N2.normal} mentioned earlier, we obtain
\begin{equation}\label{aldingi}
\begin{aligned}
 \int_{D^c} \phi^{-\frac{1}{p-1}}(x)\left|\Delta^2 \phi(x)\right|^{p^{\prime}}\mathrm{d}x&=\int_{R<|x|<2R}\left[B(x) \xi^l\left(\frac{|x|}{R}\right)\right]^{-\frac{1}{p-1}} \left|\Delta^2 \phi(x)\right|^{p^{\prime}}\mathrm{d}x \\
&\leq  CR^{-\frac{2p}{p-1}} (\ln{R})^{\frac{p}{p-1}}  \int_{R<|x|<2 R} \xi^{l-4 p^{\prime}}\left(\frac{|x|}{R}\right) \mathrm{d}x\\
&\stackrel{|x|=r}{\leq} C R^{-\frac{2p}{p-1}} (\ln{R})^{\frac{p}{p-1}}    \int_{R}^{2R}r \mathrm{d}r\\
&\leq    C R^{-\frac{2}{p-1}} (\ln{R})^{\frac{p}{p-1}} .
\end{aligned}
\end{equation}

\noindent To obtain the expected outcome, it is necessary to merge  \eqref{zeta.estimate} and \eqref{aldingi}.
\end{proof}

\subsection{The case when $\boldsymbol{N=4}$}
\begin{Lem}\label{H(x.estimate).N4.normal}
     For $R<|x|<2 R$ the following estimates are valid
     \begin{equation*}
         \begin{aligned}
             B(x) &\leq C \ln R,\\
             |\nabla B(x)| &\leq C R^{-1},\\
             |\Delta B(x)|&\leq CR^{-2},
         \end{aligned}
     \end{equation*}
 and
$$|\nabla\cdot\Delta B(x)|\leq CR^{-3}.$$
\end{Lem}
\begin{proof}
By \eqref{H(x)}, we arrive at
    $$B(x)= 2\ln{|x|} +|x|^{-2}- 1 \leq C \ln R, $$
 taking its derivative results in
$$\nabla B(x)=2x|x|^{-2}(1-|x|^{-2}), \quad |\nabla B(x)| \leq CR^{-1}.$$
Now, define the function  
    $$
h(r)=2 \ln r+r^{-2}-1,
$$
and differentiate to obtain 
$$
h^{\prime}(r)=2 r^{-1}(1-r^{-2}),$$ with the second derivative as $$h^{\prime \prime}(r)=6r^{-4}-2r^{-2},$$
which implies
$$
\begin{aligned}
\Delta B(x)& =h^{\prime \prime}(r)+\frac{3}{r} h^{\prime}(r)  =4 r^{-2}, \quad |\Delta B(x)|\leq C R^{-2}, \quad r=|x| .
\end{aligned}
$$
Then we have
$$\nabla\cdot\Delta B(x)=-\frac{8x}{|x|^4}, \quad |\nabla\cdot\Delta B(x)|\leq CR^{-3},$$
and that completes the proof.
\end{proof}

\begin{Lem}\label{estimate.t.normal.n4} 
    For $R>0$ and $T>0$ the following estimate holds:
    \begin{equation*}
         \int_Q \varphi_2^{-\frac{1}{p-1}}(x,t)\left|\partial_t \varphi_2(x,t)\right|^{p^{\prime}} \mathrm{d}x \mathrm{d}t\leq C T^{1-p^{\prime}}R^{4}\ln{R}.
    \end{equation*}
   
\end{Lem}

\begin{proof}
Drawing upon the features of the test functions  \eqref{psi1}, we arrive at

    $$
\begin{aligned}
 \int_Q  \varphi_2^{-\frac{1}{p-1}}(x,t)\left|\partial_t \varphi_2(x,t)\right|^{p^{\prime}}\mathrm{d}x \mathrm{d}t =\left(\int_{D^c} \phi(x)\mathrm{d}x \right)\left(\int_0^T \eta^{-\frac{1}{p-1}}(t)\big|\partial_t\eta(t)\big|^{p^{\prime}} \mathrm{d}t\right).
\end{aligned}
$$

\noindent For sufficiently large $R$ and for $l> \frac{4p}{p-1}$, we obtain  
\begin{equation}\label{xi.integral.N2.normal.n4}
    \begin{aligned}
        \int_{D^c} \phi(x)\mathrm{d}x&=\int_{1<|x|<2R} B(x) \xi^l\left(\frac{|x|}{R}\right)\mathrm{d}x\\
        & \leq C\ln{R}\int_{1<|x|<2R} \xi^l\left(\frac{|x|}{R}\right)\mathrm{d}x\\
        & \stackrel{|x|=r}{\leq} C\ln{R}\int_1^{2R}  r^{3} \mathrm{d}r\\
        &\leq CR^{4}\ln{R}.
    \end{aligned}
\end{equation}
\noindent  The integral with the cut-off function is provided in \eqref{N2.derivetavizeta.estimate}, combining estimates we complete the proof.
\end{proof}
The estimate for the test function has already been demonstrated in lemma \ref{samet.lemma}, and the estimate for the biharmonic function has been shown in lemma \ref{H(x.estimate).N4.normal}. Thus, combining the results of these lemmas is sufficient to establish the following lemma.

\begin{Lem}\label{samet.lemma.normal.n=4}
    For $R<|x|<2 R$, the following estimates hold:
$$\left|B(x) \Delta^2 \xi^{\ell}\left(\frac{|x|}{R}\right)\right|  \leq C R^{-4}\ln{R}\xi^{\ell-4}\left(\frac{|x|}{R}\right),$$ 
$$\left|\nabla B(x) \cdot \nabla\left(\Delta \xi^{\ell}\left(\frac{|x|}{R}\right)\right)\right|  \leq C R^{-4}\xi^{\ell-4}\left(\frac{|x|}{R}\right),$$ 
$$\left|\Delta B(x) \Delta \xi^{\ell}\left(\frac{|x|}{R}\right)\right| \leq C R^{-4}\xi^{\ell-4}\left(\frac{|x|}{R}\right),$$ and
$$\left|\nabla(\Delta B(x)) \cdot \nabla \xi^{\ell}\left(\frac{|x|}{R}\right)\right|  \leq CR^{-4}\xi^{\ell-4}\left(\frac{|x|}{R}\right),$$
where $\cdot$ denotes the inner product in $\mathbb{R}^N$.
\end{Lem}

\begin{Lem}\label{estimate.b.normal.n=4}
    The following integral can be evaluated as
    \begin{equation*}
        \int_Q  \varphi_2^{-\frac{1}{p-1}}(x,t)\left|\Delta^2 \varphi_2(x,t)\right|^{p'} \mathrm{d}x \mathrm{d}t \leq CT^1 R^{-\frac{4}{p-1}} (\ln{R})^{\frac{p}{p-1}} .
    \end{equation*}   
\end{Lem}

\begin{proof}  From test functions can be obtained following 
\begin{equation*}
\begin{aligned}
\int_Q   \varphi_2^{-\frac{1}{p-1}}(x,t)\left|\Delta^2 \varphi_2(x,t)\right|^{p'}\mathrm{d}x \mathrm{d}t=\left(\int_0^T\eta(t) \mathrm{d}t \right)\left(\int_{D^c} \phi^{-\frac{1}{p-1}}(x)\left|\Delta^2\phi(x)\right|^{p^{\prime}}\mathrm{d}x \right).
\end{aligned}    
\end{equation*}

\noindent Moreover, by \eqref{bxi} and lemma \ref{samet.lemma.normal.n=4} it is straightforward to show
\begin{equation}\label{n=4.phi1}
    \begin{aligned}
\left|\Delta^2\phi(x)\right| & \leq C\left( R^{-4}\ln{R} +R^{-4} +R^{-4}+ R^{-4}\right) \xi^{\ell-4}\left(\frac{|x|}{R}\right) \\
& \leq  C R^{-4}\ln{R}\xi^{\ell-4}\left(\frac{|x|}{R}\right), \quad R<|x|<2 R .
\end{aligned}
\end{equation}
In addition, by \eqref{n=4.phi1} and lemma \ref{H(x.estimate).N4.normal} mentioned earlier, we obtain
\begin{equation}\label{aldingi.N=4}
\begin{aligned}
 \int_{D^c} \phi^{-\frac{1}{p-1}}(x)\left|\Delta^2 \phi(x)\right|^{p^{\prime}}\mathrm{d}x&=\int_{R<|x|<2 R}\left[B(x) \xi^l\left(\frac{|x|}{R}\right)\right]^{-\frac{1}{p-1}} \left|\Delta^2 \phi(x)\right|^{p^{\prime}}\mathrm{d}x \\
&\leq  CR^{-\frac{4p}{p-1}} (\ln{R})^{\frac{p}{p-1}}  \int_{R<|x|<2 R} \xi^{l-4 p^{\prime}}\left(\frac{|x|}{R}\right) \mathrm{d}x\\
&\stackrel{|x|=r}{\leq} C R^{-\frac{4p}{p-1}} (\ln{R})^{\frac{p}{p-1}}    \int_R^{2R}r^3 \mathrm{d}r\\
&\leq    C R^{-\frac{4}{p-1}} (\ln{R})^{\frac{p}{p-1}} .
\end{aligned}
\end{equation}

\noindent To obtain the expected outcome, it is necessary to merge  \eqref{zeta.estimate} and \eqref{aldingi.N=4}.
\end{proof}

\subsection{The case when $\boldsymbol{N=3}$ and $\boldsymbol{N\geq 5}$}
\begin{Lem}\label{H(x.estimate).N3.normal}
     For $R<|x|<2 R$, we have
$$ |{}B(x)| \leq C ,$$ $$|\nabla B(x)| \leq CR^{3-N},$$ $$|\Delta B(x)|\leq CR^{2-N},
$$ and
$$|\nabla\cdot\Delta B(x)|\leq CR^{1-N}.$$
\end{Lem}
\begin{proof} By \eqref{H(x)} it obvious that
\begin{equation*}
    \begin{aligned}
        B(x)&= |x|^{2-N}-1+\frac{N-2}{N-4}\left(1-|x|^{4-N}\right)\\&=-\big(1-|x|^{2-N}\big)+\frac{N-2}{N-4}\left(1-|x|^{4-N}\right),
    \end{aligned}
\end{equation*}
therefore, we have
$$|B(x)|\leq C,$$
and

$$\nabla B(x)=x(N-2)|x|^{2-N}\left(1-|x|^{-2}\right), \quad |\nabla B(x)| \leq CR^{3-N}.$$
     Let
$$
h(r)=r^{2-N}-1+\frac{N-2}{N-4}\left(1-r^{4-N}\right),
$$
upon differentiation, we have
\begin{equation*}
\begin{aligned}
     h^{\prime}(r)&=(2-N) r^{1-N}+(N-2)r^{3-N},\\
     h^{\prime\prime}(r)&=(2-N)(1-N) r^{-N}+(N-2)(3-N)r^{2-N},
\end{aligned}
\end{equation*}
as a result
\begin{align*}
\Delta B(x)&=h^{\prime \prime}(r)+\frac{N-1}{r} h^{\prime}(r)\\&=
2(N-2)r^{2-N}\\&\leq CR^{2-N},\end{align*}
then
$$\nabla\cdot\Delta B(x)=2(N-2)(2-N)\frac{x}{|x|^{N}}, \quad |\nabla\cdot\Delta B(x)|\leq CR^{1-N},$$
which completes the proof.
\end{proof}

\begin{Lem}\label{estimate.t.normal.cr}
For $R>0$ and $T>0$ the following estimate holds:
    \begin{equation*}
        \int_Q \varphi_2^{-\frac{1}{p-1}}(x,t)\left|\partial_t \varphi_2(x,t)\right|^{p^{\prime}} \mathrm{d}x \mathrm{d}t\leq C T^{1-p^{\prime}}R^{N}.
    \end{equation*}

\end{Lem}

\begin{proof}
Employing the attributes of the test functions specified in equation \eqref{psi1}, we can derive
     $$
\begin{aligned}
  \int_Q  \varphi_2^{-\frac{1}{p-1}}(x,t)\left|\partial_t \varphi_2(x,t)\right|^{p^{\prime}} \mathrm{d}x \mathrm{d}t=\left(\int_{D^c} \phi(x)\mathrm{d}x \right)\left(\int_0^T \eta^{-\frac{1}{p-1}}(t)\big|\partial_t\eta(t)\big|^{p^{\prime}} \mathrm{d}t\right).
\end{aligned}
$$

\noindent The integral with cut-off function is the same as in \eqref{N2.derivetavizeta.estimate}. For $R>0$, the following results obtained by using the lemma \ref{H(x.estimate).N3.normal} 
\begin{equation}\label{xi.integral.N2}
    \begin{aligned}
        \int_{D^c} \phi(x)\mathrm{d}x&=\int_{1<|x|<2R} B(x) \xi^l\left(\frac{|x|}{R}\right)\mathrm{d}x\\
        & \leq C\int_{1<|x|<2R} \xi^l\left(\frac{|x|}{R}\right)\mathrm{d}x\\
        & \stackrel{|x|=r}{\leq} C\int_1^{2R}  r^{N-1} \mathrm{d}r\\
        &\leq CR^{N}.
    \end{aligned}
\end{equation}
The combination of \eqref{N2.derivetavizeta.estimate} and \eqref{xi.integral.N2} substantiates the proof.

\end{proof}

 It is straightforward to calculate the following lemma while taking into account lemma \ref{samet.lemma} and lemma \ref{H(x.estimate).N3.normal}.

\begin{Lem}\label{samet.lemma.normal}
    For $R<|x|<2 R$ the following estimates hold:
$$\left|B(x) \Delta^2 \xi^{\ell}\left(\frac{|x|}{R}\right)\right|  \leq C R^{-4}\xi^{\ell-4}\left(\frac{|x|}{R}\right),$$ $$\left|\nabla B(x) \cdot \nabla\left(\Delta \xi^{\ell}\left(\frac{|x|}{R}\right)\right)\right|  \leq C R^{-N}\xi^{\ell-4}\left(\frac{|x|}{R}\right),$$ $$\left|\Delta B(x) \Delta \xi^{\ell}\left(\frac{|x|}{R}\right)\right|  \leq C R^{-N}\xi^{\ell-4}\left(\frac{|x|}{R}\right),$$ and
$$\left|\nabla(\Delta B(x)) \cdot \nabla \xi^{\ell}\left(\frac{|x|}{R}\right)\right|  \leq CR^{-N}\xi^{\ell-4}\left(\frac{|x|}{R}\right),$$
where $\cdot$ denotes the inner product in $\mathbb{R}^N$.
\end{Lem}

\begin{Lem}\label{estimate.b.normal.cr}
    The following integral can be evaluated as
    \begin{equation*}
        \int_Q  \varphi_2^{-\frac{1}{p-1}}(x,t)\left|\Delta^2 \varphi_2(x,t)\right|^{p'} \mathrm{d}x \mathrm{d}t \leq CT^1 R^{N-4p'}.
    \end{equation*}   
\end{Lem}

\begin{proof}  From test functions can be obtained following 
     $$
\begin{aligned}
\int_Q  \varphi_2^{-\frac{1}{p-1}}(x,t)\left|\Delta^2 \varphi_2(x,t)\right|^{p'}\mathrm{d}x \mathrm{d}t =\left(\int_0^T\eta(t) \mathrm{d}t \right)\left(\int_{D^c} \phi^{-\frac{1}{p-1}}(x)\left|\Delta^2\phi(x)\right|^{p^{\prime}}\mathrm{d}x\right).
\end{aligned}
$$
\noindent The substitution of \eqref{bxi} and lemma \ref{samet.lemma.normal}, we deduce 
$$
\begin{aligned}
\left|\Delta^2\psi(x)\right|  &\leq C\left( R^{-4} + R^{-N} +R^{-N}+ R^{-N}\right) \xi^{\ell-4}\left(\frac{|x|}{R}\right),\\
& \leq C R^{-4} \xi^{\ell-4}\left(\frac{|x|}{R}\right),\quad R<|x|<2 R,
\end{aligned}
$$
which implies that

\begin{equation*}
\begin{aligned}
  \int_{D^c} \phi^{-\frac{1}{p-1}}(x)\left|\Delta^2 \phi(x)\right|^{p^{\prime}}\mathrm{d}x &=\int_{R<|x|<2R}\left[B(x) \xi^ \ell\left(\frac{|x|}{R}\right)\right]^{-\frac{1}{p-1}} \left|\Delta^2 \phi(x)\right|^{p^{\prime}}\mathrm{d}x \\
&\leq  C R^{-4p'} \int_{R<|x|<2 R} \xi^{l-4 p^{\prime}}\left(\frac{|x|}{R}\right) \mathrm{d}x\\
&\stackrel{|x|=r}{\leq} CR^{-4p'}    \int_R^{2R}r^{N-1}\mathrm{d}r\\
&\leq CR^{N-4p'} .
\end{aligned}
\end{equation*}
and from \eqref{zeta.estimate}, we get
 \begin{equation*}
        \int_Q  \varphi_2^{-\frac{1}{p-1}}(x,t)\left|\Delta^2 \varphi_2(x,t)\right|^{p'} \mathrm{d}x \mathrm{d}t \leq CT^1 R^{N-4p'}.
    \end{equation*} 
The proof has been completed.
\end{proof}

\subsection{Test functions with logarithmic arguments}

We define test functions of the form
\begin{equation}\label{psi2}
   \varphi_2(t, x)=\eta(t) \phi(x)=\zeta^{\ell}\left(\frac{t}{T}\right)B(x) \mathcal{F}^{\ell}\left(\frac{\ln \left( \frac{|x|}{\sqrt{R}}\right)}{\ln (\sqrt{R})}\right), \quad(t, x) \in Q,
\end{equation}

\noindent where the $B(x)$ is biharmonic function as \eqref{H(x)}, $\eta(t)$ is cut-off function as \eqref{zeta} and $\mathcal{F}: \mathbb{R} \rightarrow[0,1]$ is a smooth function as \eqref{f.s}.

\begin{Lem}\label{crit.2.t}
    For $N\geq 5$ and $p=\frac{N}{N-4}$ the following estimate holds:
    \begin{equation}\label{aldingi1}
         \int_Q \varphi_2^{-\frac{1}{p-1}}(x,t)\left|\partial_t \varphi_2(x,t)\right|^{p^{\prime}}\mathrm{d}x \mathrm{d}t \leq C T^{1-\frac{N}{4}}R^{N} .
    \end{equation}
   
\end{Lem}
    
\begin{proof}
Applying the properties of the test functions as defined in \eqref{psi2}, we derive
     $$
\begin{aligned}
 \int_Q \varphi_2^{-\frac{1}{p-1}}(x,t)\left|\partial_t \varphi_2(x,t)\right|^{\frac{N}{4}}\mathrm{d}x \mathrm{d}t&=\left(\int_{D^c} \phi(x)\mathrm{d}x \right)\left(\int_0^T \eta^{-\frac{1}{p-1}}(t)\big|\partial_t\eta(t)\big|^{\frac{N}{4}} \mathrm{d}t\right) .
\end{aligned}
$$

\noindent implying the lemma \ref{H(x.estimate).N3.normal}, we get
\begin{equation}\label{psi11}
    \begin{aligned}
        \int_{D^c} \phi(x) \mathrm{d}x &=\int_{1<|x|<R} B(x)  \mathcal{F}^{\ell}\left(\frac{\ln \left( \frac{|x|}{\sqrt{R}}\right)}{\ln (\sqrt{R})}\right)\mathrm{d}x \\
         & \leq C\int_{1<|x|<R} \mathcal{F}^l\left(\frac{\ln \left( \frac{|x|}{\sqrt{R}}\right)}{\ln (\sqrt{R})}\right)\mathrm{d}x\\
        & \stackrel{|x|=r}{\leq} C\int_1^{R} r^{N-1} \mathrm{d}r\\
        &\leq CR^{N}.
    \end{aligned}
\end{equation}

\noindent Therefore, \eqref{aldingi1} follows from \eqref{N2.derivetavizeta.estimate} and \eqref{psi11}.
\end{proof}
The proof of the next lemma is based on lemma \ref{salistiry} and lemma \ref{H(x.estimate).N3.normal}. Due to the biharmonic function, the estimate has little changes.
\begin{Lem}\label{salistiry.normal}
      The following estimates hold:
$$\left|\nabla(\Delta B(x)) \cdot\nabla \mathcal{F}^{\ell}\left(\frac{\ln\left(\frac{|x|}{\sqrt{R}}\right) }{\ln{\sqrt{R}}}\right)\right|  \leq   \frac{C|x|^{-N}}{\ln{\sqrt{R}}}\mathcal{F}^{\ell-4}\left(\frac{\ln\left(\frac{|x|}{\sqrt{R}}\right) }{\ln{\sqrt{R}}}\right),$$ 
$$\left|\nabla B(x) \cdot\nabla\left(\Delta \mathcal{F}^{\ell}\left(\frac{\ln\left(\frac{|x|}{\sqrt{R}}\right) }{\ln{\sqrt{R}}}\right)\right)\right|  \leq  \frac{C|x|^{-N}}{\ln{\sqrt{R}}}\mathcal{F}^{\ell-4}\left(\frac{\ln\left(\frac{|x|}{\sqrt{R}}\right) }{\ln{\sqrt{R}}}\right),$$ 
$$\left|B(x)\Delta^2 \mathcal{F}^{\ell}\left(\frac{\ln\left(\frac{|x|}{\sqrt{R}}\right) }{\ln{\sqrt{R}}}\right)\right|  \leq  \frac{C|x|^{-4}}{\ln{\sqrt{R}}} B(x) \mathcal{F}^{\ell-4}\left(\frac{\ln\left(\frac{|x|}{\sqrt{R}}\right) }{\ln{\sqrt{R}}}\right),$$ and
$$\left|\Delta B(x) \Delta \mathcal{F}^{\ell}\left(\frac{\ln\left(\frac{|x|}{\sqrt{R}}\right) }{\ln{\sqrt{R}}}\right)\right|  \leq  \frac{C|x|^{-N}}{\ln{\sqrt{R}}} \mathcal{F}^{\ell-4}\left(\frac{\ln\left(\frac{|x|}{\sqrt{R}}\right) }{\ln{\sqrt{R}}}\right),$$
where $\cdot$ denotes the inner product in $\mathbb{R}^N$.
\end{Lem}

\begin{Lem}\label{crit.2.b}
    The following integral can be estimated as
    \begin{equation}\label{aldingi3}
        \int_Q \varphi_2^{-\frac{1}{p-1}}(x,t)\left|\Delta^2 \varphi_2(x,t)\right|^{p'} \mathrm{d}x \mathrm{d}t\leq CT^1R^2 (\ln{R})^{-\frac{N}{4}}.
    \end{equation}
    
\end{Lem}

\begin{proof}

By employing the attributes of the test functions specified in equation \eqref{psi2}, we can derive    

  $$
\begin{aligned}
\int_Q  \varphi_2^{-\frac{1}{p-1}}(x,t)\left|\Delta^2 \varphi_2(x,t)\right|^{p'}\mathrm{d}x \mathrm{d}t =\left(\int_0^T\eta(t) \mathrm{d}t \right)\left(\int_{D^c} \phi_1^{-\frac{1}{p-1}}(x)\left|\Delta^2\phi(x)\right|^{p^{\prime}}\mathrm{d}x \right).
\end{aligned}
$$
Further, from \eqref{bxi} and lemma \ref{salistiry.normal} and following estimate
\begin{equation*}
    \left|\Delta^2 \phi(x)\right|\leq C\left(\frac{|x|^{-4}+|x|^{-N}}{\ln{\sqrt{R}}} \right) \mathcal{F}^{\ell-4}\left(\frac{\ln\left(\frac{|x|}{\sqrt{R}}\right) }{\ln{\sqrt{R}}}\right),
\end{equation*}

\noindent that is
\begin{equation}\label{aldingi4}
  \begin{aligned}
    \int_{D^c} \phi_1^{-\frac{1}{p-1}}(x)\big|\Delta^2 &\phi(x)\big|^{p^{\prime}}\mathrm{d}x\\
    & =   \int_{\sqrt{R}<|x|<R} B(x)^{-\frac{1}{p-1}} \mathcal{F}^{-\frac{\ell}{p-1}}\left(\frac{\ln\left(\frac{|x|}{\sqrt{R}}\right) }{\ln{\sqrt{R}}}\right)\left|\Delta^2\phi(x)\right|^{p'}\mathrm{d}x \\
    &\leq C(\ln{\sqrt{R}})^{-\frac{N}{4}}\int_{\sqrt{R}<|x|<R} |x|^{2-N} \mathcal{F}^{l-N}\left(\frac{\ln\left(\frac{|x|}{\sqrt{R}}\right) }{\ln{\sqrt{R}}}\right)\mathrm{d}x\\
    &\stackrel{|x|=r}{\leq} C(\ln{\sqrt{R}})^{-\frac{N}{4}}\int_{\sqrt{R}}^R r^{2-N}r^{N-1}\mathrm{d}r\\
    &\leq C R^2(\ln{R})^{-\frac{N}{4}}.
\end{aligned}  
\end{equation}

 \noindent Therefore, \eqref{zeta.estimate} and \eqref{aldingi4} implies  \eqref{aldingi3}. This brings the proof to its conclusion.

\end{proof}

\subsection{Properties of $\boldsymbol{\varphi_3(x,t)}$}
We consider the test function as
\begin{equation}\label{chi}
\varphi_3(x,t)=\eta(t) \chi(x)=\zeta^{\ell}\left(\frac{t}{T}\right)\xi^{\ell}\left(\frac{|x|}{R}\right), \quad(t, x) \in  Q,
\end{equation}
where $R,T$ are sufficiently large constants, $\eta$ is cut-off function \eqref{zeta} and $\xi: \mathbb{R} \rightarrow[0,1]$ is a smooth function \eqref{xi.s}.

\begin{Lem}\label{estimate.t.neu} 
    For $R>0$ and $T>0$ the following estimate holds:
    \begin{equation*}
        \int_Q \varphi_3^{-\frac{1}{p-1}}(x,t)\left|\partial_t\varphi_3(x,t)\right|^{p^{\prime}} \mathrm{d}x \mathrm{d}t\leq C T^{1-p^{\prime}}R^{N}.
    \end{equation*}
   
\end{Lem}

\begin{proof}  It's obvious that from \eqref{chi}, we have
$$
\begin{aligned}
 \int_Q \varphi_3^{-\frac{1}{p-1}}(x,t)\left|\partial_t \varphi_3(x,t)\right|^{p^{\prime}}\mathrm{d}x \mathrm{d}t =\left(\int_{D^c} \chi(x)\mathrm{d}x \right)\left(\int_0^T \eta^{-\frac{1}{p-1}}(t)\big|\partial_t\eta(t)\big|^{p^{\prime}} \mathrm{d}t\right).
\end{aligned}
$$
Moreover, \eqref{xi.s} implies 
    \begin{equation}\label{aldingi5}
    \begin{aligned}
        \int_{D^c}\chi(x) \mathrm{d}x = \int_{1<|x|<2R} \xi^l\left(\frac{|x|}{R}\right)\mathrm{d}x \stackrel{|x|=r}{\leq} C\int_1^{2R}  r^{N-1} \mathrm{d}r\leq CR^{N}.
    \end{aligned}
\end{equation}
The desired result is obtained by combining \eqref{N2.derivetavizeta.estimate} and \eqref{aldingi5}.
\end{proof}

\begin{Lem}\label{estimate.b.neu}
    The following integral can be evaluated as
    \begin{equation*}
        \int_Q  \varphi_3^{-\frac{1}{p-1}}(x,t)\left|\Delta^2 \varphi_3(x,t)\right|^{p'} \mathrm{d}x \mathrm{d}t \leq CT^1 R^{N-4p'}.
    \end{equation*}   
\end{Lem}

\begin{proof}  From test functions \eqref{chi} can be obtained following 
     $$
\begin{aligned}
\int_Q  \varphi_3^{-\frac{1}{p-1}}(x,t)\left|\Delta^2 \varphi_3(x,t)\right|^{p'}\mathrm{d}x \mathrm{d}t=\left(\int_0^T\eta(t) \mathrm{d}t \right)\left(\int_{D^c} \chi_1^{-\frac{1}{p-1}}(x)\left|\Delta^2\chi(x)\right|^{p^{\prime}}\mathrm{d}x\right).
\end{aligned}
$$
Further, taking into account lemma \ref{samet.lemma}, we have
\begin{equation*}
    \left|\Delta^2\chi(x)\right| \leq C R^{-4} \xi^{\ell-4}\left(\frac{|x|}{R}\right),
\end{equation*}
 \noindent and \eqref{zeta.estimate} implies
\begin{equation*}\label{aldingi6}
\begin{aligned}
\int_Q  \varphi_3^{-\frac{1}{p-1}}(x,t)\left|\Delta^2 \varphi_3(x,t)\right|^{p'}\mathrm{d}x \mathrm{d}t &= CT \int_{R<|x|<2R} \chi_1^{-\frac{1}{p-1}}(x)\left|\Delta^2\chi(x)\right|^{p^{\prime}}\mathrm{d}x \\
&\leq  C TR^{-4 p^{\prime}}  \int_{R<|x|<2 R} \xi^{l-4 p^{\prime}}\left(\frac{|x|}{R}\right) \mathrm{d}x\\
&\stackrel{|x|=r}{\leq} C TR^{-4 p^{\prime}}  \int_{1<r<2}r^{N-1}\mathrm{d}r\\
&\leq         CT R^{N-4 p^{\prime}} .
\end{aligned}
\end{equation*}

\noindent     The desired result achieved.
\end{proof}

\subsection{The logarithmic argument case}
We introduce test functions of the form
$$
\varphi_3(t, x)=\eta(t) \chi(x)=\zeta^{\ell}\left(\frac{t}{T}\right)\mathcal{F}^{\ell}\left(\frac{\ln \left( \frac{|x|}{\sqrt{R}}\right)}{\ln (\sqrt{R})}\right), \quad(t, x) \in Q,
$$

\noindent where $\eta(t)$ is cut-off function as \eqref{zeta} and $\mathcal{F}: \mathbb{R} \rightarrow[0,1]$ is a smooth function as \eqref{f.s}.

\begin{Lem}\label{estimate.t.neu.cr} 
    Let $N\geq 5$ and $p=\frac{N}{N-4}$. For $R>0$ and $T>0$ the following estimate holds:
    \begin{equation*}
         \int_Q \varphi_3^{-\frac{1}{p-1}}(x,t)\left|\partial_t \varphi_3(x,t)\right|^{p^{\prime}}\mathrm{d}x \mathrm{d}t \leq C T^{1-\frac{N}{4}}R^{N} .
    \end{equation*}
   
\end{Lem}
    
\begin{proof}
Utilizing the properties of the test functions described in equation \eqref{chi}, we obtain
        $$
\begin{aligned}
 \int_Q \varphi_3^{-\frac{1}{p-1}}(x,t)\left|\partial_t \varphi_3(x,t)\right|^{p^{\prime}}\mathrm{d}x \mathrm{d}t =\left(\int_{D^c} \chi(x)\mathrm{d}x \right)\left(\int_0^T \eta^{-\frac{1}{p-1}}(t)\big|\partial_t\eta(t)\big|^{p^{\prime}} \mathrm{d}t\right).
\end{aligned}
$$
The space variable integral is estimated as  
\begin{equation}\label{aldingi7}
    \begin{aligned}
        \int_{D^c} \chi(x) \mathrm{d}x &=\int_{D^c}   \mathcal{F}^{\ell}\left(\frac{\ln \left( \frac{|x|}{\sqrt{R}}\right)}{\ln (\sqrt{R})}\right)\mathrm{d}x \\
         & \leq C\int_{1<|x|<R} \mathcal{F}^{\ell}\left(\frac{\ln \left( \frac{|x|}{\sqrt{R}}\right)}{\ln (\sqrt{R})}\right)\mathrm{d}x \\
        & \stackrel{|x|=r}{\leq} C\int_{1<|x|<R}  r^{N-1} \mathrm{d}r\\
        &\leq CR^{N},
    \end{aligned}
\end{equation}
and inequality \eqref{N2.derivetavizeta.estimate} implies that 

   $$
\begin{aligned}
 \int_Q \varphi_3^{-\frac{1}{p-1}}(x,t)\left|\partial_t \varphi_3(x,t)\right|^{p^{\prime}}\mathrm{d}x \mathrm{d}t \leq C T^{1-\frac{N}{4}}R^{N} .
\end{aligned}
$$
Therefore, we complete the proof.
\end{proof}

\begin{Lem}\label{estimate.b.neu.cr}
    The following integral can be estimated as
    \begin{equation*}
        \int_Q \varphi_3^{-\frac{1}{p-1}}(x,t)\left|\Delta^2 \varphi_3(x,t)\right|^{p'} \mathrm{d}x \mathrm{d}t\leq CT (\ln{\sqrt{R}})^{1-\frac{N}{4}}.
    \end{equation*}
    
\end{Lem}

\begin{proof}
By leveraging the characteristics of the test functions, we obtain

     $$
\begin{aligned}
\int_Q  \varphi_3^{-\frac{1}{p-1}}(x,t)\left|\Delta^2 \varphi_3(x,t)\right|^{p'}\mathrm{d}x \mathrm{d}t =\left(\int_0^T\eta(t) \mathrm{d}t \right)\left(\int_{D^c} \chi_2^{-\frac{1}{p-1}}(x)\left|\Delta^2\chi(x)\right|^{p^{\prime}}\mathrm{d}x \right).
\end{aligned}
$$
the integral of cut-off function  same as \eqref{zeta.estimate} , and classifying the biharmonic integral as in  lemma \ref{salistiry}, we get
$$
\begin{aligned}
    \int_{D^c} \chi_2^{-\frac{1}{p-1}} (x)\left|\Delta^2\chi(x)\right|^{p'} \mathrm{d}x & =  CT \int_{{\sqrt{R}<|x|<R}} \mathcal{F}^{-\frac{\ell}{p-1}}\left(\frac{\ln\left(\frac{|x|}{\sqrt{R}}\right) }{\ln{\sqrt{R}}}\right)\left|\Delta^2\chi(x)\right|^{p'}\mathrm{d}x \\
    &\leq CT(\ln{\sqrt{R}})^{-\frac{N}{4}}\int_{\sqrt{R}<|x|<R} r^{-N} \mathcal{F}^{l-\frac{4p}{p-1}}\left(\frac{\ln\left(\frac{|x|}{\sqrt{R}}\right) }{\ln{\sqrt{R}}}\right)\mathrm{d}x\\
    & \stackrel{|x|=r}{\leq} CT(\ln{\sqrt{R}})^{-\frac{N}{4}}\int_{\sqrt{R}}^Rr^{-N}r^{N-1}\mathrm{d}r\\
    &\leq CT (\ln{\sqrt{R}})^{1-\frac{N}{4}},
\end{aligned}
$$
we get the desired results.

\end{proof}
\begin{Lem}\label{f(x).lem}
    For  sufficiently large $R>0$ the following evaluations appropriate:
    \begin{equation*}
        \int_Q f(x) \varphi_j(x,t)\mathrm{d}x \mathrm{d}t = CT\int_{1<|x|<2R}   f(x)A(x)\xi^{\ell}\left(\frac{|x|}{R}\right)\mathrm{d}x,
    \end{equation*}
    \noindent or 
     \begin{equation*}
        \int_Q f(x) \varphi_j(x,t)\mathrm{d}x \mathrm{d}t = CT\int_{1<|x|<R}   f(x)A(x)\mathcal{F}^{\ell}\left(\frac{\ln \left( \frac{|x|}{\sqrt{R}}\right)}{\ln (\sqrt{R})}\right)\mathrm{d}x,
    \end{equation*}
    
    where $A(x)$ is defined as \eqref{a(x).gen}.
    
\end{Lem}
\begin{proof}
    As we can see, the integral can also be divided by the variables and using equation \eqref{zeta.estimate}, we derive 

\begin{equation*}
\begin{aligned}
    \int_Q  f(x) \varphi_j(x,t)\mathrm{d}x \mathrm{d}t&=\int_0^T \eta(t) d t \int_{D^c} f(x) A(x) \xi^{\ell}\left(\frac{|x|}{R}\right)\mathrm{d}x\\
    &= CT\int_{1<|x|<2R}f(x)A(x)  \xi^{\ell}\left(\frac{|x|}{R}\right) \mathrm{d}x.
\end{aligned}   
\end{equation*}
Same for the second integral with different support of test function.
This concludes the proof.
\end{proof}

\subsection{Properties of test function $\boldsymbol{\psi}$. } Let us consider the following cut-off functions 
\begin{equation}\label{psi}
  \psi_R(x, t):=\left[\phi\left(\frac{(|x|-1)^4+t}{R}\right)\right]^{4 p^{\prime}}, \quad \psi_R^*(x, t):=\left[\phi^*\left(\frac{(|x|-1)^4+t}{R}\right)\right]^{4 p^{\prime}}, 
\end{equation}

\noindent for all $(x, t) \in D^c\times[0, \infty)$, where $\phi \in C^{\infty}([0, \infty) $ satisfies 
$$
\phi(s)=\left\{\begin{array}{ll}
1 & \text { if } s \in[0,1 / 2], \\
\searrow & \text { if } s \in(1 / 2,1), \\
0 & \text { if } s \in[1, \infty),
\end{array}\right.$$ and $$\phi^*(s)= \begin{cases}0 & \text { if } s \in[0,1 / 2), \\
\phi(s) & \text { if } s \in[1 / 2, \infty).\end{cases}
$$
\noindent  Thus, we define a test function as follows
\begin{equation}
    \psi(x, t) := A(x)\psi_R(x, t), \quad \text{for all} \quad (x, t) \in D^c\times[0, \infty),
\end{equation}
where $A(x)$ is \eqref{a(x).gen}. 

\begin{Lem}\label{lifespanes}
    Let  $\psi_R$ and $\psi_R^*$ be a test function \eqref{psi}. The following estimate holds for positive constant $C$:
    $$
    \begin{aligned}
         \left|\partial_t \psi_R(t, x)\right|&\leq C R^{-1}\left[\psi_R^*(t, x)\right]^{\frac{1}{p}}\\
\left|\Delta^2 \psi_R(t, x)\right| &\leq CR^{-1}\left[\psi_R^*(t, x)\right]^{\frac{1}{p}} . 
    \end{aligned}
  $$
\begin{proof}
    The first inequality implies immediately
    $$
\begin{aligned}
\left|\partial_t \psi_R(t, x)\right| &=4 p^{\prime} R^{-1} \left[\phi\left(s_R(t, x)\right)\right]^{4 p^{\prime}-1} \phi^{\prime}\left(s_R(t, x)\right)\\
&\leq CR^{-1} \left[\phi^*\left(s_R(t, x)\right)\right]^{\frac{4 p^{\prime}}{p}} [\phi\left(s_R(t, x)\right)]^3\left|\phi^{\prime}\left(s_R(t, x)\right)\right| \\
& \leq CR^{-1}\left[\psi_R^*(t, x)\right]^{\frac{1}{p}} .
\end{aligned}
$$
The second assertion also follows from direct differentiation of radial functions. 
\end{proof}
  
\end{Lem}

\begin{Lem}[Lemma 3.10 \cite{Ikeda1}]\label{ikedalemma}
Let $\delta>0, C_0>0, R_1>0, \theta \geq 0$ and $0 \leq w \in L_{\mathrm{loc}}^1\left([0, T) ; L^1\left(D^c\right)\right)$ for $T>R_1$. Assume that for every $R \in\left[R_1, T\right)$,

$$
\delta+\int_0^T  \int_{D^c} w(x, t) \psi_R(x, t) d x d t \leq C_0 R^{-\frac{\theta}{p^{\prime}}}\left(\int_0^T  \int_{D^c} w(x, t) \psi_R^*(x, t) d x d t\right)^{\frac{1}{p}}.
$$
Then $T$ has to be bounded above as follows:

$$
T \leq \begin{cases}\left(R_1^{(p-1) \theta}+(\log 2) C_0^p \theta \delta^{-(p-1)}\right)^{\frac{1}{(p-1) \theta}} & \text { if } \theta>0, \\ \exp \left(\log R_1+(\log 2)(p-1)^{-1} C_0^p \delta^{-(p-1)}\right) & \text { if } \theta=0 .\end{cases}
$$
\end{Lem}

\section{Proof of main Theorems}

This section is devoted to the proof of the main theorems. First, we show the nonexistence and existence of the solutions in the first critical case. Following that, we prove the second critical case as well with the existence of the solutions. Then, we show the nonexistence of the solutions when $f(x)=0.$ Finally, we demonstrate the lifespan of the solution. 
\subsection{Proof of Theorem \ref{th.n=3} to the problem I and problem VI}
The theorem will be proved by using the method of contradiction.  Let us assume that $u$ is the weak solution to either Problem I or Problem VI.

\noindent  \textbf{\textit{(i)} The case $\boldsymbol{N=2}$.}  By definition \ref{def1}, we have:
\begin{equation*}
    \begin{aligned}
        \int_Q|u(x,t)|^{p}  &\varphi_j(x,t)\mathrm{d}x \mathrm{d}t +\int_Q f(x) \varphi_j(x,t)\mathrm{d}x \mathrm{d}t +\int_{D^c} u_0(x)\varphi_j(x,0)\mathrm{d}x\\
        & =-\int_Q u(x,t)\partial_t\varphi_j(x,t)\mathrm{d}x \mathrm{d}t +\int_Q u(x,t)\Delta^2 \varphi_j(x,t)\mathrm{d}x \mathrm{d}t.
    \end{aligned}
\end{equation*}
 Due to the  $\varepsilon$-Young inequality, we derive
\begin{equation*}
    \begin{aligned}
 \int_Q & u(x,t)  \partial_t \varphi_j(x,t)\mathrm{d}x \mathrm{d}t
\\&
\leq \varepsilon_1  \int_Q|u(x,t)|^p \varphi_j(x,t)\mathrm{d}x \mathrm{d}t+c\left(\varepsilon_1\right)  \int_Q \varphi_j^{-\frac{1}{p-1}}(x,t)\left|\partial_t \varphi_j(x,t)\right|^{p^{\prime}} \mathrm{d}x \mathrm{d}t,
\end{aligned}
\end{equation*}
and
\begin{equation*}
    \begin{aligned}
 \int_Q & u(x,t) \Delta^2 \varphi_j(x,t)\mathrm{d}x \mathrm{d}t
\\ &\leq \varepsilon_2  \int_Q|u(x,t)|^p \varphi_j(x,t)\mathrm{d}x \mathrm{d}t+c\left(\varepsilon_2\right)  \int_Q\varphi_j^{-\frac{1}{p-1}}(x,t) \mid \Delta^2 \varphi_j(x,t)|^{p^{\prime}}\mathrm{d}x \mathrm{d}t.
\end{aligned}
\end{equation*}
Later selecting $\varepsilon_1=\varepsilon_2=\frac{1}{2}$ and rearranging the result, we obtain
\begin{equation}\label{after Young}
\begin{aligned}
\int_Q & f(x) \varphi_j(x,t)\mathrm{d}x \mathrm{d}t+\int_{D^c} u_0(x) \varphi_j(x,0)\mathrm{d}x\\
&\leq \int_Q \varphi_j^{-\frac{1}{p-1}}(x,t)\left|\Delta^2 \varphi_j(x,t)\right|^{p'}\mathrm{d}x \mathrm{d}t+\int_Q\varphi_j^{-\frac{1}{p-1}}(x,t)\left| \partial_t \varphi_j(x,t)\right|^{p^{\prime}}\mathrm{d}x \mathrm{d}t.
\end{aligned}   
\end{equation}
As observed, the test function \eqref{phi} satisfies the last property in definition \ref{def1} for both problem I and problem VI. Therefore, it is enough to prove the problem I. From lemma \ref{estimate.t.n=2} and lemma \ref{estimate.b.n=2}, we know the assessment of the integrals
\begin{equation*}
\begin{aligned}
     \int_Q f(x) \varphi_1(x,t)\mathrm{d}x \mathrm{d}t+\int_{D^c} u_0(x) \varphi_1(x,0)\mathrm{d}x \leq C\left(T^1 R^{2-4p'}(\ln{R})^{\frac{p}{p-1}} + T^{1-p^{\prime}}R^{2} \ln R\right),
\end{aligned}   
\end{equation*}
for $T=R^j, j>4$,  handling with lemma \ref{f(x).lem}, we can rewrite 
\begin{equation*}
\begin{aligned}
      \int_{1<|x|<2R}  f(x)A(x)\xi^{\ell}\left(\frac{|x|}{R}\right) \mathrm{d}x 
     &\leq CR^{-j} \int_{1<|x|<2R} |u_0(x)| \varphi_1(x,0)\mathrm{d}x \\&+ C     R^{2-4p'} \ln^{\frac{p}{p-1}}R  +R^{-jp^{\prime}+2}\ln R,
\end{aligned}   
\end{equation*}

\noindent where $A(x)$  is harmonic function \eqref{A(x)} for particular problem.  By taking 
$R$ sufficiently large and noting $u_0 \in L_{l o c}^1(D^c)$,  we obtain
\begin{equation*}
\begin{aligned}
      \int_{1<|x|<2R} f(x)A(x)\xi^{\ell}\left(\frac{|x|}{R}\right) \mathrm{d}x \leq C    R^{-4p'+2} \ln R   ,
\end{aligned}   
\end{equation*}

\noindent It is straightforward to observe that $2-4p'<0$ for any  $p>1$. Therefore, passing to the limit as $R \rightarrow \infty$ in the above inequality with 
\begin{equation}\label{lim.supc}
    \lim _{R \rightarrow+\infty} A(x)\xi^{\ell}\left(\frac{|x|}{R}\right)=A(x)\xi^{\ell}(0)=A(x).
\end{equation}
Therefore, we get
$$
\lim _{R \rightarrow \infty} \int_{1<|x|<2R} f(x)A(x)  \xi^{\ell}\left(\frac{|x|}{R}\right) \mathrm{d}x=\int_{D^c} f(x)A(x) \mathrm{d}x \leq 0,
$$

\noindent which contradicts our assumption
$$
\int_{D^c} f(x)A(x) dx>0 .
$$
The proof (i) is complete for $N=2$.

\noindent \textbf{ \textit{(ii)} The subcritical case.}  In this proof we deal with test function \eqref{phi}, where the only distinction arises from the harmonic function \eqref{A(x)} with $N\geq 3.$ Reiterating the process as the previous proof, we apply inequality \eqref{after Young}, lemma \ref{estimate.t} and lemma \ref{estimate.b.n=3}. Therefore, we deduce
\begin{equation*}
\begin{aligned}
    \int_Q f(x) \varphi_1(x,t)\mathrm{d}x \mathrm{d}t+\int_{D^c} u_0(x) \varphi_1(x,0)\mathrm{d}x \leq C\left(T^1 R^{N-4p'} + T^{1-p^{\prime}}R^{N} \right).
\end{aligned}   
\end{equation*}
By applying Lemma \ref{f(x).lem} and setting  $T=R^j$ with $j>4$, we derive
\begin{equation*}
\begin{aligned}
      \int_{1<|x|<2R}  f(x)A(x)\xi^{\ell}\left(\frac{|x|}{R}\right) \mathrm{d}x 
     &\leq   CR^{-j} \int_{1<|x|<2R} |u_0(x)| \varphi_1(x,0)\mathrm{d}x\\
     &+ C    \left( R^{N-4p'} +R^{-jp^{\prime}+N} \right),
\end{aligned}   
\end{equation*}
When $N\in\{3,4\}$, for any $p>1$ we have the following: 
\begin{equation*}\label{critical}
    \begin{aligned}
        & -\frac{4p}{p-1} +N < 0.     
    \end{aligned}
\end{equation*}
For $$p < \frac{N}{N-4},\,\,N\geq 5$$ it also follows that $-\frac{4p}{p-1} +N < 0.$  
        
\noindent Therefore, by \eqref{lim.supc}, we have 
\begin{equation*}
\lim _{R \rightarrow \infty} \int_{1<|x|<2R} f (x) A(x)\xi^{\ell}\left(\frac{|x|}{R}\right)d x=\int_{ D^c} f (x) A(x)d x\leq 0,
\end{equation*}

\noindent which contradicts our assumption
$
\int_{D^c} f(x)A(x) dx>0 .
$
The proof has reached its conclusion.

\noindent \textbf{\textit{(ii)} Critical case.} The proof follows a similar approach to the previous one. The main difference is using test functions with logarithmic arguments \eqref{phic}.  Using the inequality \eqref{after Young} and lemma \ref{estimate.t} and lemma \ref{estimate.b}, we have
\begin{equation*}
\begin{aligned}
    \int_Q f(x) \varphi_1(x,t)\mathrm{d}x \mathrm{d}t+\int_{D^c} u_0(x) \varphi_1(x,0)\mathrm{d}x\leq CT^1 \left( (\ln{R})^{-\frac{N}{4}+1}+R^{\frac{N}{2}\left(1-\frac{N}{2}\right)}(\ln{R})^{-\frac{N}{4}}\right), 
\end{aligned}   
\end{equation*} 
\noindent Now, we assign  $T=R^j$ and lemma \ref{f(x).lem} implies 
\begin{equation*}
    \begin{aligned}
 \int_{1<|x|<R} f(x) A(x)\mathcal{F}^{\ell}\left(\frac{\ln\left(\frac{|x|}{\sqrt{R}}\right) }{\ln{\sqrt{R}}}\right)  \mathrm{d}x \leq C &\Big( (\ln{R})^{-\frac{N}{4}+1}+R^{\frac{N}{2}\left(1-\frac{N}{2}\right)}(\ln{R})^{-\frac{N}{4}}\Big)\\
 &+CR^{-j} \int_{D^c} |u_0(x)| \varphi_1(x,0)\mathrm{d}x,
    \end{aligned}
\end{equation*}

\noindent taking $j>4$ and passing to the limit $R \rightarrow \infty$ implies that
\begin{equation}\label{lim.crit}
\lim _{R \rightarrow+\infty} A(x)\mathcal{F}^{\ell}\left(\frac{\ln\left(\frac{|x|}{\sqrt{R}}\right) }{\ln{\sqrt{R}}}\right)=A(x)\mathcal{F}^{\ell}(-1)=A(x).    
\end{equation}
Therefore, we get
$$
\lim _{R \rightarrow \infty} \int_{1<|x|<R} f(x) A(x)\mathcal{F}^{\ell}\left(\frac{\ln\left(\frac{|x|}{\sqrt{R}}\right) }{\ln{\sqrt{R}}}\right) \mathrm{d}x=\int_{ D^c} f(x)A(x) \mathrm{d}x \leq 0,
$$
which is a contradiction with $$\int_{D^c} f(x)A(x) \mathrm{d}x>0.$$ The proof has been concluded.

\subsection{Proof of Theorem \ref{th.n=3} to the problem II}
Test function $\varphi_2$ satisfies the conditions required for a weak solution to Problem II and  $A(x)$ is given by \eqref{H(x)}.

\noindent \textbf{\textit{(i)} The case $\boldsymbol{N=2}$.} By definition \ref{def1}, we have
\begin{equation*}
    \begin{aligned}
        \int_Q  |u|^{p}\varphi_2\mathrm{d}x \mathrm{d}t +\int_Q f(x) \varphi_2(x,t)\mathrm{d}x \mathrm{d}t +&\int_{D^c} u_0(x)\varphi_2(x,0)\mathrm{d}x \\
        &= -\int_Q u\partial_t\varphi_2\mathrm{d}x \mathrm{d}t +\int_Q u\Delta^2 \varphi_2\mathrm{d}x \mathrm{d}t,
    \end{aligned}
\end{equation*}
where $\varphi_2$ is given by \eqref{psi1}. The inequality  \eqref{after Young} for $\varphi_2$ takes the following form: 
\begin{equation*}
\begin{aligned}
\int_Q & f(x) \varphi_2(x,t)\mathrm{d}x \mathrm{d}t+\int_{D^c} u_0(x) \varphi_2(x,0)\mathrm{d}x\\
&\leq \int_Q \varphi_2^{-\frac{1}{p-1}}(x,t)\left|\Delta^2 \varphi_2(x,t)\right|^{p'}\mathrm{d}x \mathrm{d}t+\int_Q\varphi_2^{-\frac{1}{p-1}}(x,t)\left| \partial_t \varphi_2(x,t)\right|^{p^{\prime}}\mathrm{d}x \mathrm{d}t.
\end{aligned}   
\end{equation*}
From lemma \ref{estimate.t.normal} , lemma \ref{estimate.b.normal} and lemma \ref{f(x).lem},  we obtain the following results:
\begin{equation*}
\begin{aligned}
       \int_{1<|x|<2R} f(x) A(x)\xi^{\ell}\left(\frac{|x|}{R}\right) \mathrm{d}x &\leq \frac{1}{CT}\int_{1<|x|<2R} |u_0(x)| \varphi_2(x,0)\mathrm{d}x \\&+C     T^{-p^{\prime}}R^{4} \ln{R} +R^{-\frac{2}{p-1}}(\ln{R})^{\frac{p}{p-1}},
\end{aligned}   
\end{equation*}

\noindent where $A(x)$ is biharmonic function \eqref{H(x)}. Let us set $T=R^j$ with $j>4$, then
\begin{equation*}
\begin{aligned}
      \int_{1<|x|<2R}  f(x)A(x)\xi^{\ell}\left(\frac{|x|}{R}\right) \mathrm{d}x 
     &\leq CR^{-j}\int_{1<|x|<2R} |u_0(x)| \varphi_2(x,0)\mathrm{d}x \\&+ C      R^{4-jp^{\prime}} \ln{R} +R^{-\frac{2}{p-1}}(\ln{R})^{\frac{p}{p-1}},
\end{aligned}   
\end{equation*}

\noindent taking into account \eqref{lim.supc}, and passing to the limit as $R \rightarrow \infty$ in the above inequality one can obtain
$$
\lim _{R \rightarrow \infty} \int_{1<|x|<2R} f(x) A(x) \xi^{\ell}\left(\frac{|x|}{R}\right) \mathrm{d}x=\int_{D^c} f(x)A(x) \mathrm{d}x \leq 0,
$$

\noindent which contradicts our assumption
$$
\int_{D^c} f(x)A(x) dx>0 .
$$
The proof is now finished.

\noindent \textbf{\textit{(i)} The case $\boldsymbol{N=4}$.}  The inequality  \eqref{after Young} for $\varphi_2$ takes the following form: 
\begin{equation*}
\begin{aligned}
\int_Q & f(x) \varphi_2(x,t)\mathrm{d}x \mathrm{d}t+\int_{D^c} u_0(x) \varphi_2(x,0)\mathrm{d}x\\
&\leq \int_Q \varphi_2^{-\frac{1}{p-1}}(x,t)\left|\Delta^2 \varphi_2(x,t)\right|^{p'}\mathrm{d}x \mathrm{d}t+\int_Q\varphi_2^{-\frac{1}{p-1}}(x,t)\left| \partial_t \varphi_2(x,t)\right|^{p^{\prime}}\mathrm{d}x \mathrm{d}t.
\end{aligned}   
\end{equation*}
From lemma \ref{estimate.t.normal.n4} , lemma \ref{estimate.b.normal.n=4} and lemma \ref{f(x).lem},  we obtain the following results:
\begin{equation*}
\begin{aligned}
       \int_{D^c} f(x) A(x)\xi^{\ell}\left(\frac{|x|}{R}\right) \mathrm{d}x &\leq \frac{1}{CT}\int_{D^c} |u_0(x)| \varphi_2(x,0)\mathrm{d}x  \\& +C     T^{-p^{\prime}}R^{4} \ln{R}+R^{-\frac{4}{p-1}}(\ln{R})^{\frac{p}{p-1}},
\end{aligned}   
\end{equation*}

\noindent by setting $T=R^j$ with $j>4$ , we have
\begin{equation*}
\begin{aligned}
      \int_{1<|x|<2R}  f(x)A(x)\xi^{\ell}\left(\frac{|x|}{R}\right) \mathrm{d}x 
     &\leq   CR^{-j}\int_{1<|x|<2R} |u_0(x)| \varphi_2(x,0)\mathrm{d}x\\
     &+CR^{4-jp^{\prime}} \ln{R} +CR^{-\frac{4}{p-1}}(\ln{R})^{\frac{p}{p-1}},
\end{aligned}   
\end{equation*}

\noindent \noindent taking into account \eqref{lim.supc}, and passing to the limit as $R \rightarrow \infty$ in the above inequality one can obtain
$$
\lim _{R \rightarrow \infty} \int_{1<|x|<2R} f(x) A(x) \xi^{\ell}\left(\frac{|x|}{R}\right) \mathrm{d}x=\int_{D^c} f(x)A(x) \mathrm{d}x \leq 0,
$$

\noindent which contradicts our assumption
$$
\int_{D^c} f(x)A(x) dx>0 .
$$
The proof is now finished.

\noindent \textbf{ \textit{(ii)} The subcritical case.}  From lemma \ref{estimate.t.normal.cr} and lemma \ref{estimate.b.normal.cr} we have
\begin{equation*}\label{estimate.f}
\begin{aligned}
    \int_Q f(x) \varphi_2(x,t)\mathrm{d}x \mathrm{d}t+\int_{D^c} u_0(x)& \varphi_2(x,0)\mathrm{d}x \leq C  \left( T^1 R^{N-4p'}  +T^{1-p^{\prime}}R^{N}   \right),
\end{aligned}   
\end{equation*}

\noindent Substituting lemma \ref{f(x).lem} to the left-hand side, we get
\begin{equation*}
\begin{aligned}
      \int_{1<|x|<2R} f(x) A(x)\xi^{\ell}\left(\frac{|x|}{R}\right) \mathrm{d}x 
     &\leq  \frac{1}{CT}\int_{1<|x|<2R}  |u_0(x)| \varphi_2(x,0)\mathrm{d}x\\
     &+C\left(  R^{N-\frac{4p}{p-1}}  +T^{-p^{\prime}}R^{N}   \right),
\end{aligned}   
\end{equation*}
for $T=R^j$ with $j>4$, we derive
\begin{equation*}
\begin{aligned}
      \int_{1<|x|<2R} f(x)A(x) \xi^{\ell}\left(\frac{|x|}{R}\right) \mathrm{d}x 
     &\leq   CR^{-j}\int_{1<|x|<2R} |u_0(x)| \varphi_2(x,0)\mathrm{d}x\\
      &+C \left(  R^{N-\frac{4p}{p-1}}  +R^{N-\frac{jp}{p-1}}   \right),
\end{aligned}   
\end{equation*}
\noindent Considering the power of $R$, it obvious that when $N=3$ the power is negative for any $p>1$. When  $N\geq 5$ power is also negative if $p$ satisfies the following  
        $$p < \frac{N}{N-4}.$$

\noindent By setting $j>4$ and passing to the limit as $R \rightarrow \infty$ in the above inequality one can derive
$$
\lim _{R \rightarrow \infty} \int_{1<|x|<2R} f(x) A(x)\xi^{\ell}\left(\frac{|x|}{R}\right) \mathrm{d}x=\int_{D^c} f(x)A(x) \mathrm{d}x \leq 0,
$$

\noindent which contradicts our assumption
$
\int_{D^c} f(x)A(x) dx>0 .
$
The proof has been completed.

\textbf{\textit{(ii)} Critical case.} From the lemma \ref{crit.2.t}, lemma \ref{crit.2.b} and lemma \ref{f(x).lem} , we get 
\begin{equation*}
    \begin{aligned}
         CT\int_{1<|x|<R}   f(x)A(x)\mathcal{F}^{\ell}\left(\frac{\ln\left(\frac{|x|}{\sqrt{R}}\right) }{\ln{\sqrt{R}}}\right)\mathrm{d}x &+ \int_{1<|x|<R} u_0(x) \varphi_2(0,x)\mathrm{d}x\\
        &\leq C\left(T^1 (\ln{R})^{1-\frac{N}{4}}+T^{1-\frac{N}{4}}R^{N}\right), 
    \end{aligned}
\end{equation*}

\noindent  Therefore, assigning $T$ as $T=R^j$, 
we get
\begin{equation*}
    \begin{aligned}
 \int_{1<|x|<R} f(x)\mathcal{F}^{\ell}\left(\frac{\ln\left(\frac{|x|}{\sqrt{R}}\right) }{\ln{\sqrt{R}}}\right) \mathrm{d}x
 &\leq  CR^{-j}\int_{1<|x|<R} |u_0(x)| \varphi_2(x,0)\mathrm{d}x\\
 &+ C\left(  (\ln{R})^{1-\frac{N}{4}}+R^{-\frac{N(j-4)}{4}}\right),
    \end{aligned}
\end{equation*}

\noindent taking $j>4$ and passing to the limit as $R \rightarrow \infty$ in the above inequality one can obtain
$$
\lim _{R \rightarrow \infty} \int_{1<|x|<R} f(x)A(x) \mathcal{F}^{\ell}\left(\frac{\ln\left(\frac{|x|}{\sqrt{R}}\right) }{\ln{\sqrt{R}}}\right) dx=\int_{ D^c} f(x) A(x)\mathrm{d}x \leq 0,
$$
which is a contradiction with
$$
\int_{ D^c} f(x)A(x) \mathrm{d}x>0.
$$

\noindent The proof has been finalized.

\subsection{Proof of Theorem \ref{th.n=3} to  problems III-V}
Test function $\varphi_3$  satisfies the conditions outlined in the definition \ref{def1} for problems III-V. Therefore, we utilize $\varphi_3$ to address three problems simultaneously.

\noindent \textbf{\textit{(i)} Subcritical case.} Applying the definition of test function \ref{def1} and inequality \eqref{after Young}, we have
\begin{equation*}
\begin{aligned}
\int_Q & f(x) \varphi_3(x,t)\mathrm{d}x \mathrm{d}t+\int_{D^c} u_0(x) \varphi_3(x,0)\mathrm{d}x\\
&\leq \int_Q \varphi_3^{-\frac{1}{p-1}}(x,t)\left|\Delta^2 \varphi_3(x,t)\right|^{p'}\mathrm{d}x \mathrm{d}t+\int_Q\varphi_3^{-\frac{1}{p-1}}(x,t)\left| \partial_t \varphi_3(x,t)\right|^{p^{\prime}}\mathrm{d}x \mathrm{d}t.
\end{aligned}   
\end{equation*}

\noindent The estimates of lemma \ref{estimate.t.neu} and lemma \ref{estimate.b.neu} yields 
\begin{equation*}
\begin{aligned}
     \int_Q f(x) \varphi_3(x,t)\mathrm{d}x \mathrm{d}t+\int_{D^c} u_0(x) \varphi_3(x,0)\mathrm{d}x \leq C\left(T^1 R^{N-4p'} + T^{1-p^{\prime}}R^{N} \right),
\end{aligned}   
\end{equation*}

\noindent  By lemma \ref{f(x).lem} and  for $T=R^j, j>4$, we get
\begin{equation*}
\begin{aligned}
      \int_{1<|x|<2R}  f(x)\xi^{\ell}\left(\frac{|x|}{R}\right) \mathrm{d}x 
     &\leq   CR^{-j} \int_{1<|x|<2R} |u_0(x)| \varphi_3(x,0)\mathrm{d}x\\
     &+ C    \left( R^{N-4p'} +R^{-jp^{\prime}+N} \right),
\end{aligned}   
\end{equation*}

\noindent  taking the $R$ big enough observe that 
\begin{equation*}
\begin{aligned}
      \int_{1<|x|<2R}  f(x)\xi^{\ell}\left(\frac{|x|}{R}\right) \mathrm{d}x \leq C    R^{-4p'+N}.
\end{aligned}   
\end{equation*}

\noindent  When $N\in\{2,3,4\},$ the power of $R$ is always negative.  For $N\geq 5$, the power remains negative provided that $p$ is subcritical: 
        $$p < \frac{N}{N-4}.$$
By \eqref{lim.supc} and passing to the limit as $R \rightarrow \infty$, , we can derive the following conclusions:
$$
\lim _{R \rightarrow \infty} \int_{1<|x|<2R} f(x)  \xi^{\ell}\left(\frac{|x|}{R}\right) \mathrm{d}x=\int_{D^c} f(x) \mathrm{d}x \leq 0,
$$

\noindent which contradicts our assumption
$$
\int_{D^c} f(x) dx>0,
$$
when $A(x)=1$ in \eqref{a(x).gen}. The proof is now complete.

\noindent\textbf{\textit{(ii)} Critical case.} Using the inequality \eqref{after Young} in subcritical case and lemma \ref{estimate.t.neu.cr}, lemma \ref{estimate.b.neu.cr}, we get
\begin{equation*}
\begin{aligned}
\int_Q f(x) \varphi_3(x,t)\mathrm{d}x \mathrm{d}t+\int_{D^c} u_0 \varphi_3(0,x) \mathrm{d}x \leq C\left(T^1 (\ln{\sqrt{R}})^{1-\frac{N}{4}}+T^{1-\frac{N}{4}}R^{N}\right).
\end{aligned}
\end{equation*}

\noindent Adopting lemma \ref{f(x).lem} and taking $T=R^j$, $p=\frac{N}{N-4}$, we derive
\begin{equation*}
    \begin{aligned}
  \int_{1<|x|<R} f(x)\mathcal{F}^{\ell}\left(\frac{\ln\left(\frac{|x|}{\sqrt{R}}\right) }{\ln{\sqrt{R}}}\right) \mathrm{d}x   \leq C&\left( (\ln{\sqrt{R}})^{-\frac{N-4}{4}}+R^{-\frac{N(j-4)}{4}}\right)\\
  &+CR^{-j} \int_{1<|x|<R} |u_0(x)| \varphi_3(x,0)\mathrm{d}x,
    \end{aligned}
\end{equation*}

\noindent by considering $j>4$ and letting R approach infinity in the preceding inequality, one can derive
$$
\lim _{R \rightarrow \infty} \int_{1<|x|<R} f(x) \mathcal{F}^{\ell}\left(\frac{\ln\left(\frac{|x|}{\sqrt{R}}\right) }{\ln{\sqrt{R}}}\right) dx=\int_{ D^c} f(x) d x \leq 0,
$$
which is a contradiction with
$
\int_{ D^c} f (x)d x>0.
$ This completes the proof.

So far, we have demonstrated (i) the subcritical case and (ii) the critical case of Theorem \ref{th.n=3}.

 \textbf{\textit{(iii)} Supercritical case.} This subsection provides proof of the supercritical case to problems I-VI. Let us 
\begin{equation}\label{v}
v(x)=\varepsilon |x|^{-m} ,
\end{equation}
where $\varepsilon>0$ and $$\frac{4}{p-1} < m<N-4.$$ 
Note that the set of $m$ is nonempty because of $p>\frac{N}{N-4}.$ Following basic calculations, we obtain
\begin{align*}
\Delta^2 v(x)&=\varepsilon m(m+2)[m-N+2][m-N+4] |x|^{-m-4}\\&=\varepsilon M |x|^{-m-4},
\end{align*}

\noindent  Therefore, we have
\begin{equation*}
\begin{aligned} 
\Delta^2 v(x) - v^p(x)& =\varepsilon M |x|^{-m-4}-\varepsilon^p |x|^{-mp}\\&=\varepsilon |x|^{-m-4} \left(M-\varepsilon^{p-1}|x|^{-mp+m+4}\right),
\end{aligned}    
\end{equation*}
where $M=m(m+2)[m-N+2][m-N+4]$. Taking into account $-mp+m+4< 0$ and for sufficiently small $\varepsilon,$ we obtain 
\begin{align*}
    \Delta^2 v(x) - v^p(x)&>\varepsilon  \left(M-\varepsilon^{p-1}\right)|x|^{-m-4}\\&=f(x)>0.
\end{align*}
Moreover, by \eqref{v} we have that
\begin{equation*}
\begin{aligned} 
&\left.v(x)\right|_{\partial B_1}=\varepsilon > 0, \quad \left.\frac{\partial v(x)}{\partial\nu}\right|_{\partial B_1}=\varepsilon m > 0,
\end{aligned}    
\end{equation*}
and
\begin{equation*}
\begin{aligned} 
 &\left.-\Delta v(x)\right|_{\partial B_1}=-\varepsilon m(m-N+2)  > 0,
\end{aligned}    
\end{equation*}
\begin{equation*}
\begin{aligned} 
&\left. \frac{\partial(\Delta v(x))}{\partial\nu}\right|_{\partial B_1}= \varepsilon m(m+2)(m-N+2)> 0.
\end{aligned}    
\end{equation*}
By comparison principle $$u(x,t)\leq v(x)\,\,\, \text{for all} \,\,\,(x,t) \in Q$$ and setting $$f(x):=\Delta^2 v(x) - v^p(x)>0,$$ we derive a stationary positive super solution to problems I-VI. The proof of the (iii) Theorem \ref{th.n=3}  to problems I-VI has been concluded. The proof of Corollary \ref{Cor1} is carried out in a manner similar to the proof of Theorem \ref{th.n=3}.
\subsection{Proof of Theorem \ref{th.2}}
This subsection provides proof of the second critical exponents to problems I-VI.  The rationale behind providing a single proof for all problems lies in the fact that each of these problems shares identical first critical exponents. Therefore, the analysis and results concerning the second critical exponents apply uniformly across all six problems.

\noindent \textbf{The cases \textit{(i)}   $\boldsymbol{\omega<0$  \textbf{and \textit{(ii)}}  $0\leq \omega<\frac{4p}{p-1}.}$} Using the inequality \eqref{after Young}
\begin{equation}\label{se}
\begin{aligned}
      \int_{D^c}  f(x)\xi^{\ell}\left(\frac{|x|}{R}\right) \mathrm{d}x 
     \leq   CR^{-j} \int_{D^c} u_0(x) \varphi(x,0)\mathrm{d}x+ C_1    \left( R^{N-4p'} +R^{-jp^{\prime}+N} \right).
\end{aligned}   
\end{equation}

\noindent Evaluating the left-hand side of the above inequality 

\begin{equation}\label{f(x)xi}
\begin{aligned}
      \int_{D^c}  f(x)\xi^{\ell}\left(\frac{|x|}{R}\right) \mathrm{d}x =  &\int_{1<|x|<2R}  f(x)\xi^{\ell}\left(\frac{|x|}{R}\right) \mathrm{d}x\\
     &\geq \int_{1<|x|<R}  f(x) \mathrm{d}x\\
     &\geq C \int_{1<|x|<R} |x|^{-\omega}\mathrm{d}x\\
     & = CR^{N-\omega}.
\end{aligned}   
\end{equation}

\noindent  For sufficiently large $R$, the inequality \eqref{se} has the following form:
\begin{equation*}
    CR^{N-\omega}\leq    C_1    \left( R^{N-4p'} +R^{-jp^{\prime}+N} \right),
\end{equation*}
 we can deduce that
\begin{equation*}
    C\leq    C_1    \left( R^{\omega-\frac{4p}{p-1}} +R^{\omega-\frac{jp}{p-1}} \right),
\end{equation*}
for $j>4$ both condition \textbf{\textit{(i)}} $\omega<0$, \textbf{\textit{(ii)}} $0\leq \omega<\frac{4p}{p-1}$ equips negative power. Therefore, $R\rightarrow \infty$ gives

\begin{equation*}
    0<C\leq 0,
\end{equation*}
which contradicts our assumptions for the constant.

\noindent\textbf{The case \textit{(iii)} }$\boldsymbol{\frac{4p}{p-1}\leq \omega <N}.$
We consider the function 
\begin{equation*}
\begin{aligned} 
\hat{u}(x)=\xi |x|^{-m}, 
\end{aligned}
\end{equation*}

\noindent where $\xi>0$ and

$$\omega -4 < m < N-4,$$
we know that 
\begin{equation*}
    \Delta^2\hat{u}(x) -\hat{u}^p(x)-f(x)\geq 0,
\end{equation*}
using the estimate from the existence theorem,
\begin{equation*}
\begin{aligned}    
     &C\xi r^{-m-4} - f(x)\geq 0,
\end{aligned}    
\end{equation*}
from easy calculation we get,
\begin{equation*}
    \begin{aligned}
         f(x) &\leq C\xi r^{-m-4}  \\
    &\leq C\xi r^{-\omega +4-4}\\
    & = C\xi r^{-\omega},
    \end{aligned}
\end{equation*}
which shows that $f\in \mathcal{I}_\omega^{-} $. The proof of theorem \ref{th.2} is completed.

\subsection{Proof of Theorem \ref{th.f}} This subsection addresses the nonexistence of solutions to problems I-VI under the condition $f(x)=0$. Since the structure of the proof is analogous across all six problems, we shall illustrate the reasoning by focusing on Problem III as a representative case for $A(x)=1$. 

\noindent \textbf{\textit{(i)} Subcritical case $\boldsymbol{p<p_{Fuj}}$.} By definition \ref{def2}, we have
\begin{equation*}
    \begin{aligned}
        \int_Q  |u|^{p}\varphi_j\mathrm{d}x \mathrm{d}t +\int_{D^c} u_0(x)\varphi_j(x,0)\mathrm{d}x = -\int_Q u\partial_t\varphi_j\mathrm{d}x \mathrm{d}t +\int_Q u\Delta^2 \varphi_j\mathrm{d}x \mathrm{d}t,
    \end{aligned}
\end{equation*}
which yields 
\begin{equation*}
\begin{aligned}
\int_Q |u(x,t)|^{p}&\varphi_3(x,t)\mathrm{d}x \mathrm{d}t +\int_{D^c} u_0(x) \varphi_3(x,0)\mathrm{d}x\\
&\leq \int_Q \varphi_3^{-\frac{1}{p-1}}(x,t)\left|\Delta^2 \varphi_3(x,t)\right|^{p'}\mathrm{d}x \mathrm{d}t+\int_Q\varphi_3^{-\frac{1}{p-1}}(x,t)\left| \partial_t \varphi_3(x,t)\right|^{p^{\prime}}\mathrm{d}x \mathrm{d}t.
\end{aligned}   
\end{equation*}

\noindent By lemma \ref{estimate.t.neu} and lemma \ref{estimate.b.neu}, we derive
\begin{equation*}
\begin{aligned}
   \int_Q& |u(x,t)|^{p}\varphi_3(x,t)\mathrm{d}x \mathrm{d}t +\int_{D^c} u_0(x) \varphi_3(x,0)\mathrm{d}x \leq C\left(T^1 R^{N-4p'} + T^{1-p^{\prime}}R^{N} \right),
\end{aligned}   
\end{equation*}
Hence, for $T=R^4$, we have  
\begin{equation}\label{up}
\begin{aligned}
  \int_0^{R^4} \int_{1<|x|<2R}|u(x,t)|^{p}\varphi_3(x,t)\mathrm{d}x \mathrm{d}t + \int_{1<|x|<2R} u_0(x) \varphi_3(x,0)\mathrm{d}x \leq C R^{N-\frac{4}{p-1}},
\end{aligned}   
\end{equation}
when $1<p<1+\frac{4}{N}$, we get 
\begin{equation*}
    N-\frac{4}{p-1}<0,
\end{equation*}
\noindent  then passing to the limit $R \rightarrow \infty$ in \eqref{up}, we arrive at 
\begin{equation*}
\begin{aligned}
\lim _{R \rightarrow \infty} &\left(\int_0^{R^4} \int_{1<|x|<2R} |u(x,t)|^{p}\varphi_3(x,t)\mathrm{d}x \mathrm{d}t +\int_{1<|x|<2R} u_0 (x)\varphi_3(x,0) \mathrm{d}x\right)\\
&=\int_0^{\infty}  \int_{D^c}|u(x,t)|^{p}A(x)\mathrm{d}x \mathrm{d}t+\int_{D^c} u_0(x)A(x)\mathrm{d}x\\
&\leq  0,
\end{aligned}
\end{equation*}
which contradicts our assumption $$\int_{D^c} u_0(x)A(x) dx\geq0.$$ The proof is complete.

\noindent \textbf{\textit{(i)} Critical case $\boldsymbol{p=p_{Fuj}}$.}
If $p=p_{Fuj},$ then from \eqref{up} one obtain that 
\begin{equation*}
    \begin{aligned}
  \int_0^{R^4} \int_{1<|x|<2R}|u(x,t)|^{p}\varphi_3(x,t)\mathrm{d}x \mathrm{d}t + \int_{1<|x|<2R} u_0(x) \varphi_3(x,0)\mathrm{d}x \leq C.
\end{aligned}   
\end{equation*}
Since $\int_{D^c} u_0(x) d x \geq 0$, passing to the limit as $R \rightarrow \infty$ in the last inequality we have that
$$
\int_0^{\infty} \int_{D^c}|u|^p d x d t \leq C,
$$
which implies that $u \in L^p\left([0,\infty) \times D^c\right)$.

\noindent By definition 
\ref{def2}, we have 
\begin{equation*}
    \begin{aligned}
        \int_Q |u|^{p}\varphi_3\mathrm{d}x \mathrm{d}t +\int_{D^c} u_0(x)\varphi_3(x,0)\mathrm{d}x =-\int_Q u\partial_t\varphi_3\mathrm{d}x \mathrm{d}t +\int_Q u\Delta^2 \varphi_3\mathrm{d}x \mathrm{d}t.
    \end{aligned}
\end{equation*}
Changing $T=R^4$ and taking into account \eqref{chi}, we derive
\begin{equation*}
    \begin{aligned} 
     \int\limits_0^{R^4}  \int\limits_{D^c}|u|^{p} \varphi_3\mathrm{d}x \mathrm{d}t  &+\int\limits_{D^c}u_0(x)\varphi_3(x,0)\mathrm{d}x=-\int\limits_{\frac{R^4}{2}}^{R^4} \int\limits_{D^c} u\partial_t\varphi_3\mathrm{d}x \mathrm{d}t +\int\limits_0^{R^4} \int\limits_{R<|x|<2R} u\Delta^2 \varphi_3\mathrm{d}x \mathrm{d}t.
    \end{aligned}
\end{equation*}
\noindent Due to H\"{o}lder's inequality, lemma \ref{estimate.t.neu} and lemma \ref{estimate.b.neu}  with $p=1+\frac{4}{N}$, we get
\begin{equation*}
    \begin{aligned} 
    \int\limits_{\frac{R^4}{2}}^{R^4}\int\limits_{D^c} u\partial_t\varphi_3\mathrm{d}x \mathrm{d}t & \leq \left( \int\limits_{\frac{R^4}{2}}^{R^4}\int\limits_{D^c} |u|^p\varphi_3\mathrm{d}x \mathrm{d}t \right)^{\frac{1}{p}}\left( \int\limits_{\frac{R^4}{2}}^{R^4} \int\limits_{D^c} \varphi_3^{-\frac{1}{p-1}}|\partial_t\varphi_3|^{p'}\mathrm{d}x \mathrm{d}t \right)^{\frac{1}{p'}}\\
    & = C\left( \int\limits_{\frac{R^4}{2}}^{R^4} \int\limits_{D^c} |u|^p\varphi_3\mathrm{d}x \mathrm{d}t \right)^{\frac{1}{p}},
    \end{aligned}
\end{equation*}
and 
\begin{equation*}
    \begin{aligned} 
    \int\limits_0^{R^4} \int\limits_{D^c} u\Delta^2 \varphi_3\mathrm{d}x \mathrm{d}t &\leq \left( \int\limits_0^{R^4} \int\limits_{R<|x|<2R} |u|^p\varphi_3\mathrm{d}x \mathrm{d}t \right)^{\frac{1}{p}}\left( \int\limits_0^{R^4} \int\limits_{R<|x|<2R} \varphi_3^{-\frac{1}{p-1}}|\Delta^2 \varphi_3|^{p'}\mathrm{d}x \mathrm{d}t \right)^{\frac{1}{p'}}\\
    &= C\left( \int\limits_0^{R^4} \int\limits_{R<|x|<2R} |u|^p\varphi_3\mathrm{d}x \mathrm{d}t \right)^{\frac{1}{p}}.
    \end{aligned}
\end{equation*}
 Therefore, we have 
\begin{equation}\label{R1}
    \begin{aligned} 
    \int\limits_0^{R^4}  \int\limits_{D^c}|u|^{p} \varphi_3\mathrm{d}x \mathrm{d}t  &+\int\limits_{D^c} u_0(x)\varphi_3(x,0)\mathrm{d}x\\
     &\leq  C\left( \int\limits_{\frac{R^4}{2}}^{R^4} \int\limits_{D^c} |u|^p\varphi_3\mathrm{d}x \mathrm{d}t \right)^{\frac{1}{p}}+C\left( \int\limits_0^{R^4} \int\limits_{R<|x|<2R} |u|^p\varphi_3\mathrm{d}x \mathrm{d}t \right)^{\frac{1}{p}}.
    \end{aligned}
\end{equation}
Since $u \in L^p_{loc}\left([0,\infty) \times D^c\right),$ the integral property implies
\begin{equation*}
    \begin{aligned} 
      \lim _{R \rightarrow \infty} \left( \int\limits_{\frac{R^4}{2}}^{R^4} \int\limits_{D^c} |u|^p\varphi_3\mathrm{d}x \mathrm{d}t \right)^{\frac{1}{p}}& = \lim _{R \rightarrow \infty}\left( \int\limits_{0}^{R^4} \int\limits_{D^c} |u|^p\varphi_3\mathrm{d}x \mathrm{d}t -\int\limits^{\frac{R^4}{2}}_{0} \int\limits_{D^c} |u|^p\varphi_3\mathrm{d}x \mathrm{d}t \right)^{\frac{1}{p}}\\
      & = \lim _{R \rightarrow \infty} \left( \int\limits_{0}^{\infty} \int\limits_{D^c} |u|^p\varphi_3\mathrm{d}x \mathrm{d}t -\int\limits^{\infty}_{0} \int\limits_{D^c} |u|^p\varphi_3\mathrm{d}x \mathrm{d}t \right)^{\frac{1}{p}}\\
      & = 0,
    \end{aligned}
\end{equation*}

\noindent and 
\begin{align*}
     \lim _{R \rightarrow \infty}  \Big( \int\limits_0^{R^4} \int\limits_{R<|x|<2R} &|u|^p\varphi_3\mathrm{d}x \mathrm{d}t \Big)^{\frac{1}{p}}\\
     & = \lim _{R \rightarrow \infty}  \left( \int\limits_0^{R^4} \int\limits_{0<|x|<2R} |u|^p\varphi_3\mathrm{d}x \mathrm{d}t -\int\limits_0^{R^4} \int\limits_{0<|x|<R} |u|^p\varphi_3\mathrm{d}x \mathrm{d}t \right)^{\frac{1}{p}}\\
      & = \lim _{R \rightarrow \infty}  \left( \int\limits_0^{\infty} \int\limits_{0<|x|<\infty} |u|^p\varphi_3\mathrm{d}x \mathrm{d}t -\int\limits_0^{\infty} \int\limits_{0<|x|<\infty} |u|^p\varphi_3\mathrm{d}x \mathrm{d}t \right)^{\frac{1}{p}}\\
      & = 0.
\end{align*}

\noindent Hence, passing to the lim $R\rightarrow\infty$ in \eqref{R1}, we get 
\begin{align*}
\lim _{R \rightarrow \infty}\Big(\int\limits_0^{R^4}  \int\limits_{D^c}|u|^{p} \varphi_3\mathrm{d}x \mathrm{d}t+\int\limits_{D^c} u_0(x)&\varphi_3(x,0)\mathrm{d}x\Big)\\
&=\int\limits_0^{\infty}  \int\limits_{D^c}|u|^{p}A(x)\mathrm{d}x \mathrm{d}t+\int\limits_{D^c} u_0(x)A(x)\mathrm{d}x\\&\leq  0,\end{align*} which contradicts the assumption $\int_{D^c} u_0(x)A(x)\mathrm{d}x\geq 0.$ The proof of the theorem is completed.

\subsection{Proof of Theorem \ref{lf}} In this subsection we give an estimate for the lifespan of the solution to Problems I-VI when $f(x)=0$. The method of proof works for all problems and it's enough to prove one of the problems.

  Without loss of generality, we may assume $R_0:=2 r_0^2<\text{LifeSpan} (u_\varepsilon)$. Then, from definition of weak solution \ref{def2}, we obtain
$$
\begin{aligned}
\int_0^T  \int_{D^c}|u(t, x)|^p \psi(t, x) \mathrm{d} x \mathrm{d} t&+\varepsilon \int_{D^c} u_0(x) \psi(0,x)\mathrm{d} x \\
\quad & \leq \int_0^T \int_{D^c}|u(t, x)|\left(\left|\partial_t \psi(t, x)\right|+\left|\Delta^2 \psi(t, x)\right|\right) \mathrm{d} x \mathrm{d} t .
\end{aligned}
$$

Now we assume that $\operatorname{LifeSpan}(u)>R_0$. From lemma \ref{lifespanes} and noting that $A(x)$ is independent of $t$, it follows  that
\begin{equation*}
\begin{aligned}
\Big|\partial_t\big(A(x) \psi_R(t)\big)&-\Delta^2\left(A(x)\psi_R(t,x)\right)\Big| \\
& = \left|\partial_t\left(A(x) \psi_R(t,x)\right)-4\nabla A(x)\nabla\Delta \psi_R(t,x)-A(x)\Delta^2 \psi_R(t,x) \right|
\\
& \leq \frac{2 C}{R} A(x)\left[\psi_R^*(t,x)\right]^{\frac{1}{p}}+\frac{C}{R} A(x)\left[\psi_R^*(t,x)\right]^{\frac{1}{p}}+\frac{C}{R} A(x)\left[\psi_R^*(t,x)\right]^{\frac{1}{p}}\\
& \leq \frac{C}{R} A(x)\left[\psi_R^*(t)\right]^{\frac{1}{p}},
\end{aligned}
\end{equation*}
making some algebraic improvements and noting $u_0(x)\in L^1_{loc}(D^c)$
$$
\begin{aligned}
c_0\varepsilon+\int_0^T  \int_{D^c}|u(t, x)|^p&A(x) \psi_R(t, x) \mathrm{d} x \mathrm{d} t \\
\quad & \leq C R^{-1} \int_0^T \int_{D^c}|u(t, x)|A(x)\left[\psi_R^*(t, x)\right]^{\frac{1}{p}} \mathrm{~d} x \mathrm{d} t \\
& \leq  CR^{-1+\frac{N+4}{4p^{\prime}}}\left(\int_0^T \int_{D^c}|u(t, x)|^p A(x)\psi_R^*(t, x) \mathrm{d} x \mathrm{d} t\right)^{\frac{1}{p}},
\end{aligned}
$$
which is
\begin{align*}
c_0\varepsilon +\int_0^T  \int_{D^c}|u(t, x)|^p &A(x)\psi_R(t, x) \mathrm{d} x \mathrm{d} t \\& \leq  CR^{-\frac{1}{p'}\left(\frac{1}{p-1}-\frac{N}{4}\right)}\left(\int_0^T \int_{D^c}|u(t, x)|^pA(x) \psi_R^*(t, x) \mathrm{d} x \mathrm{d} t\right)^{\frac{1}{p}},
\end{align*}
using lemma \ref{ikedalemma} with $w(x,t)=|u(x,t)|^p A(x)$ and 
$\theta=\frac{1}{p-1}-\frac{N}{4},$
we have the desired results.

\section{Related problems and possible prospects}
In this section we will discuss possible extensions, related problems, and prospects of problems I-IV.
\subsection{Polyharmonic heat equations}
Instead of the biharmonic heat equation \eqref{main}, one can study the following polyharmonic heat equation
\begin{equation}\label{polyhar}
u_t+(-\Delta)^m u=|u|^{p}+f(x),  \text { in }  D^c \times(0, \infty),
\end{equation} with suitable boundary conditions. Here $m\in\mathbb{N}.$ Since in case $D^c\equiv\mathbb{R}^N$ the critical exponent of \eqref{polyhar} is $\frac{N}{N-2m}$ (see \cite{majdoub}), then in the case of the exterior domain one should also expect that the critical exponent of equation \eqref{polyhar} with suitable boundary conditions will be $$p_{Crit}=\frac{N}{N-2m}.$$
Accordingly, in case $f=0$, following \cite{Gala} it is expected that the critical exponent of equation $$u_t+(-\Delta)^m u=|u|^{p},  \text { in }  D^c \times(0, \infty)$$ with suitable boundary conditions will be $$p_{Crit}=1+\frac{2m}{N}.$$

However, proving these facts requires rather cumbersome calculations. At least we assume so, based on our approaches used in this paper. We hope that in the future there will be other approaches with which it will be possible to prove the critical behavior of solutions to exterior problems for \eqref{polyhar}.
\subsection{$\boldsymbol{f(x)}$ as a boundary function}
In the classic paper by Bandle, Levine and Zhang \cite{BLZ}, in addition to problem \eqref{aaa}, were also studied the heat equation $$u_t-\Delta u=|u|^p,  \text { in } 	D^c \times(0, \infty),$$ with inhomogeneous Dirichlet $$u(x,t)=f(x)\,\,\text{in}\,\,\partial D,$$ or with inhomogeneous Neumann $$\frac{\partial u}{\partial n}(x,t)=f(x)\,\,\text{in}\,\,\partial D,$$ boundary conditions.

It has been proven that for the above problems with $$\int\limits_{\partial D}f(x)dS>0$$ the critical exponent is still $$p_{Crit}=\frac{N}{N-2}.$$ Following this idea, one can study the biharmonic heat equation $$u_t+\Delta^2 u=|u|^p,  \text { in } 	D^c \times(0, \infty),$$ with suitable inhomogeneous boundary conditions. One might expect that the critical exponents of these problems will still be $$p_{Crit}=\frac{N}{N-4},\,n\geq 5\,\,\,\, \text{and} \,\,\,\,p_{Crit}=\infty,\,n=2, 3, 4.$$

One can also consider a system of biharmonic heat equations
\begin{align*}
\left\{\begin{array}{l}u_t+\Delta^2 u=|v|^p,  \text { in } 	D^c \times(0, \infty),\\{}\\
v_t+\Delta^2 v=|u|^q, \text { in } 	D^c \times(0, \infty),\end{array}\right.
\end{align*}
with inhomogeneous or homogeneous boundary conditions. Note that recently Jleli and Samet in \cite{Sametbi} studied a similar problem for the system of biharmonic wave equations.

\subsection{System of biharmonic heat equations}
It also makes sense to study the exterior problems to the system of the biharmonic heat equation
\begin{align*}
\left\{\begin{array}{l}u_t+\Delta^2 u=|v|^p+f(x),  \text { in } 	D^c \times(0, \infty),\\{}\\
v_t+\Delta^2 v=|u|^q+g(x), \text { in } 	D^c \times(0, \infty),\end{array}\right.
\end{align*} where $p>1,\,q>1$ and $f,g$ are given source functions.

With high probability, the critical exponent of this system will be
$$N-4-\max\left\{\frac{4(p+1)}{pq-1}, \frac{4(q+1)}{pq-1}\right\},\,N\geq 5.$$

One can also study the critical properties of more general systems of the forms 
\begin{align*}\left\{\begin{array}{l}
u_t+(-\Delta)^{m_1} u=|u|^{p_1}|v|^{p_2}+f(x),  \text { in } 	D^c \times(0, \infty),\\{}\\
v_t+(-\Delta)^{m_2} v=|u|^{q_1}|v|^{q_2}+g(x), \text { in } 	D^c \times(0, \infty),\end{array}\right.
\end{align*}
or
\begin{align*}\left\{\begin{array}{l}u_t+(-\Delta)^{m_1} u=|u|^{p_1}+|v|^{p_2}+f(x),  \text { in } 	D^c \times(0, \infty),\\{}\\
v_t+(-\Delta)^{m_2} v=|u|^{q_1}+|v|^{q_2}+g(x), \text { in } 	D^c \times(0, \infty),\end{array}\right.
\end{align*}
where $m_1,m_2\in \mathbb{N}$ and $p_1>1, p_2>1, q_1>1, q_2>1.$

\section*{Acknowledgment}
The authors are very grateful to Professor Qi S. Zhang and Professor Bessem Samet for useful comments and valuable advice. Authors also would like to thank to Professor Filippo Gazzola and Professor Masahiro Ikeda for useful comments and for giving us some references.

\section*{Declaration of competing interest}
	The Authors declare that there is no conflict of interest.

\section*{Funding}  
This research has been funded by the Science Committee of the Ministry of Education and Science of the Republic of Kazakhstan (Grant No. AP23483960). NT was
supported by the collaborative research program ‘Qualitative analysis for nonlocal and fractional models’ from Nazarbayev University.

\end{document}